\documentclass[12pt]{amsart}
\usepackage[margin=2.5 cm]{geometry}
\usepackage{amsmath}
\usepackage{graphicx}
\usepackage{booktabs}
\usepackage{amsfonts}
\usepackage{amssymb}
\usepackage{xcolor}
\renewenvironment{proof}{{\bfseries Proof.}}{\qed}
\numberwithin{equation}{section} 
\theoremstyle{theorem}
\newtheorem{theorem}{Theorem}[section] 
\newtheorem{prop}[theorem]{Proposition} 
 
\newtheorem{lemma}[theorem]{Lemma}

\theoremstyle{definition}
\newtheorem{defn}[theorem]{Definition} 
\newtheorem{rk}[theorem]{Remark}

\newcommand{\G}{$\Gamma$ }

\newcommand{\tb}{\textbf}

\newcommand{\imp}{\Rightarrow}
\newcommand{\z}{\mathbb{Z}}
\newcommand{\C}{\mathbb{C}}

\newcommand{\Pic}{$\mathcal{P} $ }
\newcommand{\p}{\mathfrak{p}}

\newcommand{\Z}{\mathbb Z}

\newcommand{\secref}[1]{Section~\ref{#1}}
\newcommand{\thmref}[1]{Theorem~\ref{#1}}
\newcommand{\lemref}[1]{Lemma~\ref{#1}}
\newcommand{\remref}[1]{Remark~\ref{#1}}
\newcommand{\propref}[1]{Proposition~\ref{#1}}

\newcommand{\eqnref}[1]{~{\textrm(\ref{#1})}}

\newcommand{\dd}{\mathfrak{d}}
\newcommand{\B}{\mathcal{B}}
\newcommand{\Pb}{\mathbb{P}}

\DeclareMathOperator{\tr}{tr}

\numberwithin{equation}{section}

\usepackage[backref]{hyperref}

\begin{document}
	\title{Reversibility and its asymptotic counting in Picard group}
\author{Debattam Das}  
\address{Indian Institute of Technology Kanpur, Kanpur 208016, Uttar Pradesh, India}
\email{debattam123@gmail.com, debattam@iitk.ac.in }
    
\author{Krishnendu Gongopadhyay}
	\address{Indian Institute of Science Education and Research (IISER) Mohali,
		Knowledge City,  Sector 81, S.A.S. Nagar 140306, Punjab, India}
	\email{ krishnendu@iisermohali.ac.in}

    \author{Anirban Mukhopadhyay}
    \address{Institute of Mathematical Sciences, IV Cross Road, CIT Campus Taramani, Chennai 600 113 Tamil Nadu, India}
    \email{anirban@imsc.res.in}
\begin{abstract}
We investigate reversible elements in the Picard modular group $\mathrm{PSL}(2,\mathbb{Z}[i])$. We show that reversibility coincides with strong reversibility for Kleinian groups, in particular for the Picard group. We classify reversible elements in the Picard group and characterize loxodromic reversible elements up to conjugacy. We prove that each such conjugacy class contains exactly eight special representatives. We also obtain asymptotic estimates for the number of reversible conjugacy classes with bounded trace.
\end{abstract}
	\maketitle

    \section{Introduction}\label{intro}
The study of reversible (or real) elements in groups has a long history and arises naturally in several 
contexts, including geometric group theory, dynamical systems, and the theory of 
discrete subgroups of Lie groups, see, \cite{OS} for an overview. In the setting of Fuchsian 
groups, reversibility admits a rich geometric interpretation: hyperbolic reversible 
elements correspond to closed geodesics invariant under orientation-reversing 
symmetries, and their conjugacy classes can be identified with infinite dihedral 
subgroups.

We recall here that an element \( g \) of a group \( G \) is called \emph{reversible} (or \emph{real}) 
if it is conjugate to its inverse, i.e., there exists \( h \in G \) such that
\(
hgh^{-1} = g^{-1}.
\)
If such a conjugating element (reverser) \( h \) can be chosen to be an involution, then \( g \) 
is said to be \emph{strongly reversible} (or \emph{strongly real}). 
Reversibility is a conjugacy-invariant property, so we can say the conjugacy class containing the reversible elements are reversible classes. And strongly reversible elements 
are precisely those that can be expressed as products of two involutions.

The modular group \( \mathrm{PSL}(2,\mathbb{Z}) \) provides a particularly rich setting in which reversibility intertwines with number theory and geometry. For a non-involutory reversible element in a Fuchsian group, every reverser is itself an involution; in particular, reversibility implies strong reversibility. The reversible elements of the modular group, known as \emph{reciprocal elements}, were systematically studied by Sarnak \cite{sarnak}, who established both a classification of these elements and precise asymptotic counting results, relating them to closed geodesics passing through cone points of order two on the modular surface \(\mathbb{H}^2 / \mathrm{PSL}(2,\mathbb{Z})\). These closed geodesics are known as \textit{reciprocal geodesics}. In a different direction, Bourgain--Kontorovich studied low-lying reciprocal geodesics in \cite{bk}.
These results have also been extended in several directions, including to general  Fuchsian groups, eg.  \cite{es},  and to negatively curved Riemannian orbifolds, see \cite{paulin1}.

More recently, a combinatorial perspective based on word length has been developed. 
Basmajian and Suzzi Valli \cite{combigrowth} initiated the study of reciprocal elements in 
the modular group via free product structures, replacing hyperbolic length by 
word length. This approach was subsequently extended to Hecke groups 
\( \mathbb{Z}/2\mathbb{Z} * \mathbb{Z}/q\mathbb{Z} \) in 
\cite{das2024combinatorial, das2025asymptotic, bmv}, leading to asymptotic counting 
results for reciprocal classes in that setting.

\medskip

In contrast to the Fuchsian case, much less is known for Kleinian groups, where 
the geometry is three-dimensional and the boundary dynamics are substantially 
more intricate. A natural and fundamental example is the \emph{Picard modular group}
\(
\mathcal{P} := \mathrm{PSL}(2,\mathbb{Z}[i]),
\)
which acts discretely on hyperbolic \(3\)-space \( \mathbb{H}^3 \). 
Understanding reversibility in $\mathcal{P}$ is a problem of independent interest and was explicitly raised as an open problem in \cite{OS}. It was also pointed out in \cite{sarnak} for general Bianchi groups, of which the Picard group is the simplest example.

\medskip

The first goal of this paper is to resolve this problem by giving a complete 
classification of reversible elements in \( \mathcal{P} \). 
Our first observation is of a general nature:  for Kleinian groups, reversibility and strong reversibility coincide. Thus every reversible element of a Kleinian group admits an involutory reverser. We then specialize to the Picard modular group and obtain an explicit classification of its reversible elements. Our classification is expressed in terms of a natural invariant 
associated to loxodromic elements.

\medskip

Let \( A \in \mathcal{P} \) be loxodromic, represented by a matrix
\[
\widetilde{A}=
\begin{bmatrix}
a & b\\
c & d
\end{bmatrix},
\qquad c \neq 0.
\]
We associate to \( A \) the quantity
\begin{equation} \label{fa}
f_A := \frac{a-d}{c}.
\end{equation}
This is independent of the choice of lift, and geometrically equals the sum of the two fixed points of \(A\). Note that $f_A$ is not conjugacy invariant.

\medskip

Our first main result provides a complete classification of reversible elements 
in the Picard modular group.

\begin{theorem}\label{thm6}
Let \( g \in \mathcal{P} \). Then \( g \) is reversible if and only if one of the 
following holds:
\begin{enumerate}
\item \( g \) is elliptic;
\item \( g \) is parabolic;
\item \( g \) is loxodromic and conjugate to an element \( h \) such that
\[
f_h \in \{0,\,1,\,i,\,1+i\}.
\]
Furthermore, each reversible class of loxodromic elements with infinite cyclic centralizer in the Picard group contains exactly eight such elements.
\end{enumerate}
\end{theorem}

\medskip
The asymptotic counting of strongly reversible closed geodesics on negatively curved orbifolds has been developed in considerable generality by Parkkonen and Paulin \cite{paulin1}. Their framework, however, does not apply directly to the Picard modular group, because the relevant fixed-point sets have infinite skinning measure. Related counting arguments appear in \cite{paulin2} in the setting of ambiguous classes. In the present work, we therefore exploit instead the explicit arithmetic structure of \(\mathrm{PSL}(2,\mathbb{Z}[i])\) to study the asymptotic growth of reversible conjugacy classes.

The second goal of this paper is to determine the asymptotic growth of reversible loxodromic conjugacy classes in the Picard modular group. The classification above reduces this counting problem to an explicit arithmetic problem over the Gaussian integers. More precisely, reversible loxodromic elements admit special representatives of the forms described in Theorem~\ref{thm6} and Lemma~\ref{3.4}, and counting such representatives with bounded trace leads naturally to shifted divisor sums over \(\mathbb Z[i]\).

To obtain the required asymptotics, we establish a Gaussian-integer analogue of Ingham's classical additive divisor theorem \cite{ingham1927}, for an arbitrary fixed shift \(\kappa\in\mathbb Z[i]\setminus\{0\}\). The case \(\kappa=1\) also follows from the more general framework of Parkkonen and Paulin \cite[Theorem~2]{paulin2}. Applying this divisor-sum estimate, we obtain in Theorem~\ref{thm:BT} an asymptotic formula for the number of reversible elements in special normal form. Finally, the analysis of centralizers, together with Lemma~\ref{lem: exact}, determines the number of such special representatives lying in a given reversible conjugacy class, allowing us to pass from the count of special representatives to the asymptotic count of reversible conjugacy classes.
\begin{theorem}\label{Rcount}
    Let $R_T$ be the set of all reversible loxodromic classes where the norm of the trace in each conjugacy class is at most $T$. Then
    \[ |R_T|=\frac{3\pi}{32G}T(\log T)^{2}+O\bigl(T\log T(\log\log T)^{2}\bigr),\]
   where $G=L(2,\chi_{-4})
   $ is Catalan's constant.
    
\end{theorem}
Recently, it is shown that Catalan's constant is irrational in \cite{sun2026catalan}.

\medskip

The paper is organized as follows. In Section~\ref{prelim}, we recall 
preliminaries on Kleinian groups and the Picard modular group. In \secref{sec3}, we showed the reversible and strongly reversible are the same in the case of Kleinian group.
In \secref{sec4}, we prove the classification theorem. In \secref{sec5}, we deduced the centralisers of the Picard group. In Section~\ref{sec:ingham} we establish the
analog of Ingham's additive divisor theorem over $\mathbb{Z}[i]$
(Theorem~\ref{thm:ingham}). In Section~\ref{sec:applying} we deduce the asymptotics of
the restricted divisor sum $S(T)$ governing the matrices with $c$ coprime to $2$. In
Section~\ref{sec:ramified} we treat the complementary matrices with $(1+i)\mid c$ and
assemble the constant, completing the proof of Theorem~\ref{thm:BT}.

\section{Preliminaries}\label{prelim}

In this section, we recall some standard facts concerning the action of \(\mathrm{PSL}(2,\mathbb{C})\) on hyperbolic \(3\)-space \(\mathbb{H}^3\), following \cite{MR698777} and \cite{martelli2016introduction}. We also recall the classification of its elements and outline basic properties of the Picard modular group that will be needed in the later sections.
We use the upper half-space model
\[
\mathbb{H}^3=\{(z,t)\in \mathbb{C}\times \mathbb{R}: t>0\},
\]
whose boundary at infinity is naturally identified with
\[
\partial \mathbb{H}^3=\mathbb{C}\cup\{\infty\}.
\]
The group \(\rm PSL(2,\mathbb{C})\) acts on \(\mathbb{H}^3\cup \partial \mathbb{H}^3\) by Möbius transformations
\[
z\mapsto \frac{az+b}{cz+d},
\qquad
\begin{pmatrix}
a & b\\
c & d
\end{pmatrix}\in \rm SL(2,\mathbb{C}),
\]
where two matrices differing by sign determine the same element of \(\rm PSL(2,\mathbb{C})\). 
A discrete subgroup of \(\rm PSL(2,\mathbb{C})\) is called a \emph{Kleinian group}, and it acts properly discontinuously on $\mathbb H^3$.

A basic example of Kleinian group, which is not torsion free, is the Picard modular group
$\mathcal{P}:=\mathrm{PSL}(2,\mathbb{Z}[i])$,
where \(\mathbb{Z}[i]\) denotes the ring of Gaussian integers. Also, the Picard group is the simplest group in the class of Bianchi groups $\mathrm{PSL}(2,\mathcal{O}_d)$ where $d\in \mathbb N$, and $\mathcal O_d$ is the ring of integers in $\mathbb Q(\sqrt{-d})$. The units in \(\mathbb{Z}[i]\) are
\[
\mathbb{Z}[i]^\times=\{\pm 1,\pm i\}.
\] For \(c\in \mathbb{Z}[i]\setminus\{0\}\), we write \((c)\) for the principal ideal generated by \(c\). Accordingly, congruences modulo \(c\) are understood in the ring \(\mathbb{Z}[i]/(c)\). We denote  $a-b\in (c)$ in $\mathbb{Z}[\textit{i}]$ by  $a\equiv b \mod{c}$.  

For each element \(g\in \mathrm {PSL}(2,\mathbb{Z}[i])\), there are exactly two lifts in \(\mathrm{SL}(2,\mathbb{Z}[i])\), namely \(\widetilde g\) and \(-\widetilde g\).  Let
\[
A=\begin{pmatrix}
a & b\\
c & d
\end{pmatrix}\in \rm{SL}(2,\mathbb{C}).
\]
The fixed points of the corresponding Möbius transformation on \(\partial \mathbb{H}^3\) are the solutions of
\[
\frac{az+b}{cz+d}=z,
\]
equivalently,
\[
cz^2+(d-a)z-b=0.
\]
In particular, if \(c\neq 0\), then the sum of the two fixed points is
$\dfrac{a-d}{c}$. 

This motivates the notation $f_A$ at \eqnref{fa} in \secref{intro}. The element \(f_A\) is well defined in the loxodromic elements in \(\mathrm{PSL}(2,\mathbb{Z}[i])\), since replacing \(\widetilde A\) by \(-\widetilde A\) does not change the value of \((a-d)/c\). Moreover, when \(A\) is not parabolic and \(c\neq 0\), the element \(f_A\) is equal to the sum of the two fixed points of \(A\) on \(\partial\mathbb{H}^3\).

We recall the standard classification of elements of $\mathrm{PSL}(2,\mathbb{C})$ via their fixed points on $\partial \mathbb{H}^3$. An element $g$ is \emph{elliptic} if it fixes a point in $\mathbb{H}^3$, \emph{parabolic} if it has exactly one fixed point on $\partial \mathbb{H}^3$, and \emph{loxodromic} if it is non-elliptic and has exactly two fixed points on $\partial \mathbb{H}^3$.  In particular, every non-parabolic element has two fixed points on $\partial \mathbb{H}^3$. Loxodromic elements are also referred to as hyperbolic elements sometimes in the literature (see~\cite{martelli2016introduction}). 

Choosing a lift of $g$ to $\mathrm{SL}(2,\mathbb{C})$, elliptic and loxodromic elements are semisimple, whereas parabolic elements are not. Up to conjugacy, a loxodromic element is represented by a diagonal matrix with eigenvalues $r e^{i\theta}$ and $r^{-1} e^{i\theta}$, where $r>1$ and $0 \le \theta < \pi$, while a parabolic element is represented by an upper triangular unipotent matrix.

If an element \(g\in\rm  PSL(2,\mathbb{C})\) fixes \(\infty\), then it is represented on \(\mathbb{C}\cup\{\infty\}\) by an affine map
\[
z\mapsto az+b,
\qquad a,b\in \mathbb{C},\ a\neq 0.
\]
This observation will be used repeatedly in Section~3 after conjugating one fixed point of an element to \(\infty\). 

The fixed points of loxodromic elements in Kleinian groups are closely tied to the structure of the group. The following theorem makes this relationship precise.	
\begin{theorem}\cite[Theorem 4.3.5]{MR698777}\label{2.2.1}
    \begin{enumerate}
        \item Two M\"{o}bius transformations $g$ and $h$ have a common fixed
point in $\widehat{\mathbb{C}}=\mathbb C\cup \{\infty\}$ if and only if $trace[g, h] = 2$.
\item If $g$ and $h$ (neither the identity) have a common fixed point in $\widehat{\mathbb{C}}$ then either:\begin{enumerate}
    \item  $[g, h] = I $(so $gh = hg$) and $Fix(g) = Fix(h)$; or
\item $[g, h]$ is parabolic (and $gh\neq hg$) and $Fix(g) \neq Fix(h).$ 
\end{enumerate}

    \end{enumerate}
\end{theorem}

\section{Reversibility In Kleinian group}\label{sec3}
In this section, we show that the reversible classes and the strongly reversible classes are same in Kleinian group. 


\begin{prop}\label{pro}
	Let $ \Gamma $ be a Kleinian group (possibly with involutions). Let $ g $ be an element of $ \Gamma $. Then the following statements are equivalent:
	\begin{enumerate}
		\item $ g $ is a reversible.
	\item $ g $ is a strongly reversible.
	\end{enumerate}
\end{prop}
\begin{proof}
	First, we consider $ g $ is an involution, then the lemma will be a trivial case. Now, we assume, it is not an involution. By conjugating $ \Gamma $ by some element in $ \rm PSL(2,\mathbb{C}) $ such that $ \infty \in Fix(g) $. Then we can restrict the domain of $ g $ to $ \mathbb{C} $. Therefore, $ g $ will be an automorphism of $ \mathbb{C} .$ So, from the theory of the automorpshim of $ \mathbb{C} $, $ g $ is one of the forms of $ z\mapsto az+b $ maps, where $ a,b\in \mathbb{C} $ and $ a\neq 0. $
	 In the Kleinian group, $ g $ can be one of the three types, i.e., parabolic, elliptic, loxodromic. Now, we assume that there be an element $ h $ in \G such that 
	$$ h^{-1}gh=g^{-1}.$$  Since the fixed point set of $ g  $ contains $ \infty $.\\
	\textbf{Case 1:}\\
	First, we assume that $ g $ is a parabolic element, this implies $ Fix(g)=\{\infty\} $ and $ b\neq 0 $.
	$$h^{-1}gh=g^{-1}$$ 
	$$\imp g(h(\infty))=h(\infty)$$ 
	$$\imp h(\infty)=\infty.$$ 
That means $ h $ is also in the form of $ z\mapsto \alpha z+\beta $ where, $ \alpha,\beta\in \mathbb{C}. $ If $ h $ is also a parabolic element then, $ g $ will be an involution, that will contradicts \thmref{2.2.1}. Then, it should be loxodromic or elliptic. But, $ Fix(g)\cap Fix(h)\neq \emptyset $. This implies it must not be a loxodromic element otherwise, it will contradict the discreteness of the group. Then, $ h $ must be an elliptic element. So, there be a smallest positive integer $ n $ such that, $ h^{n}(z)=z $ for all $ z\in \mathbb{C}. $ That means,
$$h(z)=\alpha z+\beta$$ 
$$\imp h^{2}(z)=\alpha^{2}z+\alpha \beta+\beta$$ 
$$\vdots$$ 
$$\imp h^{k}(z)=\alpha^{k}z+(\alpha^{k-1}+\alpha^{k-2}+\alpha^{k-3}+\dots +\alpha+1)\beta $$ 
$$\vdots$$ 
$$\imp h^{n}(z)=\alpha^{n}z+(\alpha^{n-1}+\alpha^{n-2}+\alpha^{n-3}+\dots +\alpha+1)\beta=z \hbox{ for all $z\in \mathbb{C}.  $}$$ 
Then, $ \alpha^{n}=1 $. That means, $ h(z)=e^{i2\pi/n}z+\beta $ also, $ h^{-1}(z)=(z-\beta)e^{-i2\pi/n}. $Therefore,
\begin{equation}\label{2.1}
	h^{-1}gh(z)=az+\frac{(a-1)\beta}{\alpha}+\frac{b}{\alpha}
\end{equation}
$$\imp g^{-1}(z)=az+(a-1)e^{-2i\pi/n}\beta+e^{-2i\pi/n}b$$ 
$$\imp \frac{(z-b)}{a}=az+(a-1)e^{-2i\pi/n}\beta+e^{-2i\pi/n}b$$ 
for every $ z\in \mathbb{C}. $ This implies, $ a^{2}=1 $ and $ \frac{-b}{a}=(a-1)e^{-2i\pi/n}\beta+e^{-2i\pi/n}$.
If $ a=1 $, then $ e^{2i\pi/n}=-1 \imp n=2.$ If $ a=- 1 $, then $ g(z)=-z+b \imp g^{2}(z)=z.$ This is contradiction as we choose $ g $ to be parabolic. Then, $ h(z)=-z+\beta $. Therefore, $ h $ is an involution.\\
\textbf{Case 2:}\\
Now, assume $ g $ is not a parabolic element. Then it has two fixed points at the boundary, as earlier we assume one of the fixed point of $ g $ is $ \infty .$ Then with the similar argument, $ h $ will either fixed the fixed points of $ g $ or interchange them.

If $ h $ fixes both points, then by \thmref{2.2.1} the elements $ g$ and $h $ commute, that forces $g$ to be an involution, which contradicts the hypothesis.

So, $ h $ have to interchange the fixed points of $ g $. And also, there is at least one fixed point of $ h $ in the boundary. This implies, $ h^{2} $ has three fixed points in the boundary, i.e., $ h^{2}=Id. $ This proves the proposition.  
\end{proof}
\begin{rk}
   In Kleinian groups, parabolic elements and elliptic elements of any order can be reversible, as can involutions and loxodromic elements. That is not the case for elliptic and parabolic elements to be reversible in a Fuchsian group, (see \cite{MR4748467}).
\end{rk}

\section{Reversibility In Picard modular group}\label{sec4}

In this section, we classify the reversible elements of the Picard modular group.
	Since the Picard group is a Kleinian group, the reciprocal classes are the same as the strongly reciprocal classes. To classify the strongly reciprocal elements, it suffices to identify the involutions present in the Picard group up to conjugacy.
	We know the conjugacy classification of the involutions of Picard group \Pic  by the following theorem:
	\begin{theorem}\cite{nn}\label{3.1}
		There are only five conjugacy classes of elliptic elements in 
		$\mathcal{P}$, four for those of order $ 2 $ and one for those of order $3.$ In particular, any elliptic transformation of order 2 is conjugate to one of 
		\begin{enumerate}
			\item $ u_{2,1}:z\mapsto -z ,$
			\item $ u_{2,2}:z\mapsto -z+1, $
			\item $ u_{2,3}:z\mapsto-z+i, $
			\item$ u_{2,4}:z\mapsto-z+1+i, $
		\end{enumerate}
		while any elliptic transformation of order $3$ is conjugate to 
		$$u_{3}:z\mapsto\frac{-1}{z+1}.$$ 
	\end{theorem}
	Let $ \pi: \rm{SL}(2,\z[\it{i}])\longrightarrow\rm{PSL}(2,\z[\it{i}]) $ be the quotient map. We will discuss the reversibility of the Picard group \Pic, by looking at its lifts in $\rm{SL}(2,\z[\it{i}])$, i.e., the linear representations of the elements in \Pic. For each element $g\in \rm PSL(2,\z[\it{i}])$, there are two lifts $\tilde{g}$ and $-\tilde{g}$. We will choose of of these lifts and work in the linear group.

	Also, from now onwards, we will write the involutions $u_{2,j}$ as $u_{j}$ throughout the chapter.
	\begin{lemma}\label{3.4}
		An element $ g\in$ \Pic ~ is reciprocal if and only if $ g $ satisfies one of the following conditions:
		\begin{enumerate}
			\item $ g $ is an elliptic element.
            \item $g$ is a parabolic element.
			\item $ g $ is a loxodromic, and is conjugate to one of the following forms,$$ \begin{bmatrix}
				d&&b\\
				c&&d
			\end{bmatrix}, \begin{bmatrix}
				c+d&&b\\
				c&&d
			\end{bmatrix},~\begin{bmatrix}
				ic+d&&b\\
				c&&d
			\end{bmatrix},~\begin{bmatrix}
				(1+i)c+d&&b\\
				c&&d
			\end{bmatrix} ~~ ,$$
		\end{enumerate}
		where $ b,c,d\in \z[i] $.
	\end{lemma}
	\begin{proof}
		First, we choose $g$ as elliptic element in \Pic. It is trivial if $g$ is an involution. If $g$ is a non-involution element in Picard group $\mathcal P$, i.e., it is of order $3$, then it is conjugate to $u_3$ from \thmref{3.1}. Also $u_3$ is a product of two involutions which are $z\mapsto 1/z$ and $z\mapsto \frac{-z}{z+1}.$ Therefore every elliptic elements are reversible in Picard group.

		If we consider $g$ as a parabolic element in \Pic. The parabolic elements in $\mathcal{P}$ are of the form $z\mapsto z+k$, up to conjugacy, where $k\in \z[i]$. They can be seen as the product of two involutions, $z\mapsto -z$ and $z\mapsto-z+k$. Then it is clear that any parabolic element is reciprocal and can be seen as the form $\begin{bmatrix}
			1&& k\\ 0&& 1
		\end{bmatrix}$ upto conjugacy.

		Finally, we choose $g$ as loxodromic element in $\mathcal{P}$, then one of the lifts of $ g $  in $ \rm SL(2,\mathbb{Z}[\it{i}]) $ is   Consider $g$ to be reciprocal with reverser $h$. That means $hgh^{-1}=g^{-1}.$ Now by \thmref{3.1}, we know $h=\eta u_j\eta^{-1}$ for some $\eta\in \mathcal P$ and $j$.  
	That implies $\eta^{-1} g\eta$ as a reversible element with the reverser $u_j$ for some $j$, then
		$$u_{j}\eta^{-1}g\eta u_{j}^{-1}=\eta^{-1}g^{-1}\eta.$$ Also, one of the lifts of $\eta^{-1} g\eta$ and $ u_{j} $ in $\rm SL(2,\z[\it{i}])$ are in the form of $ \begin{bmatrix}
			a&&b\\
			c&&d
		\end{bmatrix}$ and$ \begin{bmatrix}
			i&&-ki\\
			0&&-i
		\end{bmatrix} $ respectively, where $k =0,1,i$ or $i+1.$
		That means
		$$\begin{bmatrix}
			i&&-ki\\
			0&&-i
		\end{bmatrix}\begin{bmatrix}
			a&&b\\
			c&&d
		\end{bmatrix}\begin{bmatrix}
			i&&-ki\\
			0&&-i
		\end{bmatrix}^{-1}=\pm \begin{bmatrix}
			a&&b\\
			c&&d
		\end{bmatrix}^{-1}$$
		$$\imp\begin{bmatrix}
			i&&-ki\\
			0&&-i
		\end{bmatrix}\begin{bmatrix}
			a&&b\\
			c&&d
		\end{bmatrix}=\pm\begin{bmatrix}
			d&&-b\\
			-c&&a
		\end{bmatrix}\begin{bmatrix}
			i&&-ki\\
			0&&-i
		\end{bmatrix} $$
		$$\imp \begin{bmatrix}
			ai-cki&&bi-dki\\
			-ci&&-di
		\end{bmatrix}=\pm \begin{bmatrix}
			di&&bi-dki\\
			-ci&&-ai+cki
			
		\end{bmatrix}$$
		\tb{Case 1:}
		$$ \begin{bmatrix}
			ai-cki&&bi-dki\\
			-ci&&-di
		\end{bmatrix}=- \begin{bmatrix}
			di&&bi-dki\\
			-ci&&-ai+cki
			
		\end{bmatrix}.$$
		Then $c=0$  and $ai=-di +cki\imp a+d=0.$ This implies $ g $ must be an involution which is contradiction to our assumption.\\
		\tb{Case 2:}
		$$ \begin{bmatrix}
			ai-cki&&bi-dki\\
			-ci&&-di
		\end{bmatrix}= \begin{bmatrix}
			di&&bi-dki\\
			-ci&&-ai+cki
			
		\end{bmatrix}.$$
		This implies that $ ai-cki=di\imp k=\dfrac{a-d}{c} .$ So, the lifts of $ \eta^{-1}g\eta$ are in the form of $\pm \begin{bmatrix}
			d+ck&&b\\
			c&&d
		\end{bmatrix}$, where $k$ is either $0,1,i$ or $i+1$.
	\end{proof}
\subsection{Proof of \thmref{thm6}}
Since the elliptic and parabolic cases are trivial, we may consider $g$ as a non-elliptic element in $\mathcal{P}$. By  \lemref{3.4}, we obtain, if $g\in\mathcal{P}$ is reversible then $g$ is conjugate to those elements which are of the from $\begin{bmatrix}
			d+ck&&b\\
			c&&d
		\end{bmatrix} $, where $k\in\{0,1,i,i+1\}.$  So, the theorem follows from the function $f_A$ by applying it to the matrices.
        
\qed
\begin{rk}
    Let $\gamma$, being one of the matrices from the previous theorem, also satisfy the following condition,
    \[
    J_k^{-1}\gamma J_k=\gamma^t, \qquad\text{ where } J_k=
    \begin{bmatrix}
       k&&1\\
       1&&0
    \end{bmatrix}~\text{and}~ k\in\{0,1,i,i+1\}~.
    \]The converse is also true. The proof is elementary but it shows how much these matrices close to symmetric matrices.
\end{rk}

	\section{Centraliser of the reciprocal elements in Picard group}\label{sec5}
	In this section, we have deduced the centraliser of elliptic and loxodromic elements. Let $g$ be the element in the group $G$, then the centraliser of $g$, i.e., $$ C(g)=\{~h\in G~|~ gh=hg~\}.$$ The set $C(g)$ will be from a subgroup of $ G$. In the following lemma, we have shown that the centraliser of involutions in the Picard modular group will always form a Klein-$4$ group. 
	\begin{lemma}\label{4.1}
		If $g$ is an elliptic element of order $2$ in the Picard group $\mathcal{P}$, $C(g)$ is a Klein-$4$ group.
	\end{lemma}
	
	\begin{proof}
To prove the lemma, we first need to recall the involutions $u_i$'s from \thmref{3.1} that those are only four involutions up to conjugacy. So, it suffices to show that the lemma is true for these four elements. The lifts of the involutions $u$ in $\rm SL(2,\mathbb{C}$), are of the form 
$U=\begin{bmatrix}
	-i & i k\\
	0 & i
\end{bmatrix}$, where $k\in \{0,1,i,i+1\}.$

Let we choose $g$ as an element in $C(u)$, then one of the lifts of $ g $  in $ \rm SL(2,\mathbb{Z}[\it{i}]) $ is  $ \tilde{g}=\begin{bmatrix}
	a&&b\\
	c&&d
\end{bmatrix}.$
Then,
$$ \begin{bmatrix}
	a&&b\\
	c&&d
\end{bmatrix}
\begin{bmatrix}
	-i & i k\\
	0 & i
\end{bmatrix}
=\pm \begin{bmatrix}
	-i & i k\\
	0 & i
\end{bmatrix}
\begin{bmatrix}
	a&&b\\
	c&&d
\end{bmatrix}
$$
$$\imp \begin{bmatrix}
	-ia & ib+iak \\
	-ic & id+ick
\end{bmatrix}=\pm \begin{bmatrix}
	-ia+ick & -ib+idk \\
	ic & id
\end{bmatrix}.$$
First, we consider 
$$ \begin{bmatrix}
	-ia & ib+iak \\
	-ic & id+ick
\end{bmatrix}= \begin{bmatrix}
	-ia+ick & -ib+idk \\
	ic & id
\end{bmatrix}.$$
In this case, we get, $c=-c\imp c=0.$ So, $ad=1$ in $\mathbb{Z}[i].$ Since, there are only four units in $\z[i]$, then the choices of $a$ would be from this set $\{1,-1,i,-i\}$ and we get the values of $d$ accordingly. Also, $$ib+iak=-ib+idk$$
$$\imp 2b=(d-a)k.$$
From the Table 1, $\tilde{g}$ will be any of the matrices 
$$ \pm Id, \qquad \pm U=\pm \begin{bmatrix}
	-i & i k\\
	0 & i
\end{bmatrix}$$

\begin{table}[h]\label{tb:5}
	
	\begin{tabular}{|c|c|c|}
		\hline
		$a$ & $d$ & $b$ \\
		\hline
		$1$ & $1$ & $0$ \\
		\hline
		$-1$ & $-1$ & $0$ \\
		\hline
		$i$ & $-i$ & $-ki$ \\
		\hline
		$-i$ & $i$ & $ki$ \\
		\hline
	\end{tabular}
	\caption{Values of $d$ and $b$ given constraints $ad=1$ and $k(d-a)=2b$}
\end{table}

 Now we consider, 
 $$ \begin{bmatrix}
-ia & ib+iak \\
-ic & id+ick
\end{bmatrix}=- \begin{bmatrix}
-ia+ick & -ib+idk \\
ic & id
\end{bmatrix}.$$
From the $(1,2)$ entries of the both matrices, we get $$ ib+iak= ib-idk. $$ 
$$\imp a=-d. $$ And from the $(2,2)$ entries of the both matrices, we get $$ -id= id+ick \imp ck=-2d.$$ If $k=0$ then, $d=0=a$. This implies $ad-bc=1\imp bc=-1.$ Again, there are $4$ units in $\mathbb{Z}[i].$
 Therefore the matrices are,  $\pm \begin{bmatrix}
 	0 & i \\
 	i & 0
 \end{bmatrix},\pm  \begin{bmatrix}
 0 & -1\\
 1 & 0
 \end{bmatrix}$ in the case of $k=0.$ Therefore,  For $k = 0$, i.e.,  the elements of the $C(u_{1})$ are
 
  $$\pm \begin{bmatrix} 0& i \\ i & 0 \end{bmatrix},\pm \begin{bmatrix} 0 & -1 \\ 1 & 0\end{bmatrix}, u_1, Id.$$
  Now we can assume $k\neq 0.$ Now we have the table,

	\begin{table}[h]
		\begin{tabular}{|c|c|c|c|c|}
			\hline
			$k$ & $d$ & $a$ & $b$ & $c$ \\
			\hline
			$1$ & $1$ & $-1$ & $1$ & $-2$ \\
			\hline
			$1$ & $-1$ & $1$ & $-1$ & $2$ \\
			\hline
			$1$ & $i$ & $-i$ & $0$ & $-2i$ \\
			\hline
			$1$ & $-i$ & $i$ & $0$ & $2i$ \\
			\hline
			$i$ & $1$ & $-1$ & $i$ & $2i$ \\
			\hline
			$i$ & $-1$ & $1$ & $-i$ & $-2i$ \\
			\hline
			$i$ & $i$ & $-i$ & $0$ & $-2$ \\
			\hline
			$i$ & $-i$ & $i$ & $0$ & $2$ \\
			\hline
		$1+i$ & $1$ & $-1$ & $1+i$ & $-1+i$ \\
		\hline
		$1+i$ & $-1$ & $1$ & $-(1+i)$ & $1-i$ \\
		\hline
		$1+i$ & $i$ & $-i$ & $0$ & $i(-1+i)$ \\
		\hline
		$1+i$ & $-i$ & $i$ & $0$ & $1+i$ \\
			\hline
		\end{tabular}
		\caption{Values of $a$, $b$, $c$ for given $k$ and $d$, satisfying $a=-d$, $-2d=ck$, and $-d^2-bc=1$}
	\end{table}
	 Now for $k = 1$, i.e.,  the elements of the $C(u_{2})$ are
	 $$\pm \begin{bmatrix} -1 & 1 \\ -2 & 1 \end{bmatrix}, \pm \begin{bmatrix} -i & 0 \\ -2i & i \end{bmatrix}, u_{2},Id.$$
	 
	 For $k = i$, i.e.,  the elements of the $C(u_{3})$ are
	 $$\pm \begin{bmatrix} -1 & i \\ 2i & 1 \end{bmatrix}, \pm \begin{bmatrix} -i & 0 \\ -2 & i \end{bmatrix}, u_3, Id.$$
	 
For $k = i+1$, i.e.,  the elements of the $C(u_{4})$ are
	 $$\pm \begin{bmatrix} -1 & 1+i \\ i-1 & 1 \end{bmatrix},\pm \begin{bmatrix} -i & 0 \\ i(i-1) & i \end{bmatrix}, u_4, Id.$$

Thus, for each case of $k$, the centralisers of the involutions are Klein-$4$ group.

	\end{proof}

		We need \lemref{4.1} to deduce the number of special reversible elements appears in a reversible class. Also we need this following classical theorem.
        \begin{theorem}[\textbf{H. Sato}]\cite{MR1643761}\label{ele}
        
 { The elementary Kleinian groups $G$ with two limit points, that is, the elementary Kleinian groups containing loxodromic (hyperbolic) transformations.}
\begin{enumerate}
\item $\mathrm{II}_1$: {A loxodromic (hyperbolic) cyclic group.}
\item $\mathrm{II}_2: G = \langle z \mapsto Kz,\ z \mapsto e^{2\pi i/n} z \rangle ~~ \left(|K| \neq 1,\ n \geq 2\right)$.
\item $\mathrm{II}_3: G = \langle z \mapsto Kz,\ z \mapsto 1/z \rangle~~  (|K| \neq 1)$.
\item $\mathrm{II}_4$: $G = \langle z \mapsto Kz,\ z \mapsto e^{2\pi i/n} z,\ z \mapsto 1/z \rangle$  $(|K| \neq 1,\ n \geq 2)$.
\end{enumerate}

        \end{theorem}
         We now deduce the centraliser of a loxodromic element by following lemma.
        \begin{lemma}Let $g$ be a primitive loxodromic element in \Pic. Then The centraliser of $g$ is either $\z$ or $\z\oplus\z/3\z.$ Moreover, there are at most three primitive loxodromic elements up to conjugacy which has centraliser as  $\z\oplus\z/3\z.$

        \end{lemma}
        \begin{proof}
        It is easy to see that the cyclic group generated by $g$, is a subgroup of $C(\gamma)$. Now, let $h\in C(g)$ such that $h\notin\langle g\rangle$. Then $h$ either preserve the fixed points of $g$ or interchange it. Since $h$ is arbitrary in $C(g)\setminus\langle g\rangle$, then $C(g)$ forms a elementary Kleinian group. Also, attracting and repealing points of $hgh^{-1}$ is not the same as $g$ when $h$ interchange the fixed points of $g$. Then from the classification of elementary group for Kleinian groups \thmref{ele}, the centraliser $C(g)$ is either $\mathrm{II}_2$ or $\mathrm{II}_1$, i.e., 
        \[ C(g)=\langle g\rangle\oplus T\] where, $T$ is either trivial or a finite cyclic  generated by $h$.

        Our first claim is $T=\langle h\rangle\neq\z/2\z$. On contrary, if it is true then $g\in C(h)$. That contradicts \lemref{4.1}.
         Now, our claim is there are at most three primitive reversible loxodromic elements $g$ upto conjugacy such that $T=\z/3\z$. There is only one order 3 element upto conjugacy, that commutes with $g$. That means form \cite[Theorem 4.3.6]{MR698777} and  discreteness there are only three primitive elements $g$ with $Fix(g)=Fix(h)$, namely $g,~gh$ and
$gh^{2}$ (as $h^{3}=Id$). Therefore there are at most three primitive elements having
$\z\oplus\z/3\z$ as centraliser up to conjugacy. Note, however, that every power of such
an element has the same centraliser, so the exceptional classes are not finite in number;
what is true, and all that the counting below requires, is that only $O(\log T)$ of them
have $N\mu\le T$, the traces along a cyclic group growing geometrically.

        \end{proof}

	\begin{lemma}\label{lem: exact}
    There are exactly $8$ special elements conjugate in each loxodromic reversible classes in $\mathcal{P}$ with centraliser infinite cyclic group.
\end{lemma}

	\begin{proof}
	 Let us consider $\gamma$ to be a reversible element in $\mathcal{P}$. It is easy to check that $\gamma=\delta^q$ for some nonzero integer $q$ is a reversible loxodromic element if and only if $\delta$ is a reversible loxodromic element. Thus without loss of generality, we further consider $\gamma$ as primitive element. Then we have
		\begin{equation}\label{eq:7.1}
			S^{-1}\gamma S = \gamma^{-1} 
		\end{equation}
		for some involution $S.$ It is easy to see that any reverser of $\gamma $ is in $C(\gamma)\cdot S.$ 
		Hence,
		\begin{equation}\label{eq:7.2}
			\gamma^{k} S \gamma^{k} = S 
		\end{equation}
		for any nonzero $k\in\z$.
		Since $S$ is conjugate to one of the $u_i$'s, which we denote by $u$. That is,
		\begin{equation}\label{eq:7.3}
			\eta^{-1} S \eta = u \qquad \text{for some $\eta\in \mathcal P.$}
		\end{equation}
		
		Thus $\eta_S$ is one of the solutions of \eqref{eq:7.3}. Hence $ S\eta_{S}$ is also a solution. Now we have
        
\[
\eta_s^{-1}S\eta_s\,
\eta_s^{-1}\gamma\eta_s\,
\eta_s^{-1}S\eta_s
=
\eta_s^{-1}\gamma^{-1}\eta_s.
\]
\[
\Longrightarrow
u\,\eta_s^{-1}\gamma\eta_s\,u
=
\eta_s^{-1}\gamma^{-1}\eta_s.
\]
Note that the reverses of the special forms are one of the \(u_i\)'s.
Therefore, the above relation shows that $\eta_s^{-1}\gamma\eta_s$ is a special form.
		
		That implies, two loxodromic elements are special type elements, $\eta_{S}^{-1}\gamma\eta_{S}$ and the inverse of it.

        Every reverser other than $S$, lies in $\{\gamma^nS~|~n\in \z\setminus\{0\} \}$. Thus we now choose $S'=\gamma^{2k}S$. This implies
	
	$$
	\gamma^{-k} S' \gamma^{k} = \gamma^k S\gamma^k=S.
	$$
	That means we have	$$
	\eta_{S'}^{-1} S' \eta_{S'} = u,
	$$ for some $\eta_{S'}\in $ \Pic.
	
	Therefore,
	$$
	\eta_S \eta_{S'}^{-1} \gamma^{k} S \gamma^{-k} \eta_{S'} \eta_S^{-1} = S .
	$$

This implies,	$$
	\eta_S \eta_{S'}^{-1} \gamma^{k} \in C(S).
	$$
    We know that $C(S)$ is a \text{ Klein-$4$ group}, i.e., there exists another involution $t$ such that  $C(S)=\{Id,t,S,tS\}$ and $tS$ is also involution.
	Let
	$$
	\eta_S \eta_{S'}^{-1} \gamma^{k} = \alpha\in C(S).
	$$

	Equivalently,
	$$
	(\eta_{S'} \eta_{S}^{-1})^{-1} \gamma^{k} = \alpha .
	$$
	$$\imp \gamma^{-k}\eta_{S'} \eta_{S}^{-1}=\alpha$$
	$$\imp \eta_S'=\gamma^{k}\alpha \eta_{S}$$
	
	Therefore,
	\begin{align*}
		\eta_{S'}^{-1} \gamma \eta_{S'}
		&= \eta_S^{-1} (\alpha^{-1} \gamma^{k}) \gamma \gamma^{-k} \alpha \eta_S \\
		&= \eta_S^{-1} \alpha^{-1} \gamma \alpha \eta_S.
	\end{align*}
	This will give us two more special forms i.e., $\eta_S^{-1} t^{-1} \gamma t \eta_S$ and $\eta_S^{-1} t^{-1} \gamma^{-1} t \eta_S$ other than $\eta_S^{-1} \gamma \eta_S$ and $\eta_S^{-1} \gamma^{-1} \eta_S$.

	Therefore, there are exactly four such elements.
	
		\medskip
	
	\noindent

		Now it remains to show the odd case. Now we claim that if
		$S'= \gamma^{2k+1} S$
		and $S'$ and $S$ are conjugate to the same involution $u$.
		
		Then 
	$
		\eta_S^{-1} \gamma \eta_S, \text{ and }
		\eta_{S'}^{-1} \gamma \eta_{S'}$
		are distinct for any $ \eta_{S} $, $ \eta_{S'} $.
		
		To the contrary, suppose they are not. Then
		$$
		\eta_S^{-1} \gamma \eta_S
		=
		\eta_{S'}^{-1} \gamma \eta_{S'}
		\quad \Rightarrow \quad
		\eta_S \eta_{S'}^{-1} \gamma \eta_{S'} \eta_S^{-1}
		= \gamma .
		$$
		
		Thus,
		$$
		\eta_{S'} = b \eta_S,
		\qquad
		b \in C(\gamma).
		$$
		
		Moreover,
		$$
		\eta_S^{-1} S \eta_S
		=
		\eta_{S'}^{-1} S' \eta_{S'}
		\quad \Rightarrow \quad
		b S b^{-1} = S'.
		$$
		Also,
		$$
		(\gamma^{n} S)^2 = {Id}
		\quad
		$$
		for any $n\in \z.$ That means $bS$ is an involution. 
		Moreover,
		$$
		bS= S b^{-1} = S' b
		\quad \Rightarrow \quad
		S b^{-2} = S'
		\quad \Rightarrow \quad
		{\, S' = b^{2} S \,},
		\qquad b \in C(\gamma).
		$$
		This contradicts our assumption.
        Now fix a reverser $ S $. We can divide the class of reversers in terms of $ [S] $ and $ [\gamma S] $ as reverser of each class will generate the same set of special elements where $[S]$ means the class of reversers differ by even powered of an element in $C(\gamma).$
		
		Therefore, there are exactly $8$ special elements conjugate to $\gamma.$		
	\end{proof}

\begin{rk}
    Note that for any nonzero integer $k$, the number reversible elements in the conjugacy class of a loxodromic element $\gamma$ in \Pic is the same as the number reversible elements in the conjugacy class of  $\gamma^k$.
\end{rk}

\section{Counting of Reversible conjugacy classes }\label{sec6}
{In this section, we have deduced the asymptotic growth of the number of reversible classes in Picard group $\mathcal{P}.$ The forms mentioned in \lemref{3.4} are some kind of special forms, which helps us to deal with the counting problem. We named it as `special forms'.} 

For two real valued functions $f$ and $g$ defined in $\mathbb{N}$, we write $f(T) \sim g(T)$  if the ratio of $f(T)$ and $g(T)$ approches to $1$ as $T$ tends to infinity.

For $T>0,$ we define
$$
\mathcal B_T
=\left\{
\begin{pmatrix}
d+ck & b\\
c & d
\end{pmatrix}\in \textrm{$SL(2,\z[\textit{i}])$}
:\;
4<|2d+ck|^{2}\le T,\ ~k=0,1,i,i+1
\right\}.
$$
Thus $\mathcal B_T$ is the set of all reversible special forms whose trace $\mu=2d+ck$
satisfies $4<N\mu\le T$. The lower cut-off is harmless: an element of
$\mathrm{PSL}(2,\C)$ is loxodromic precisely when its trace does not lie in the real
segment $[-2,2]$, so the traces excluded by $N\mu\le4$ which still give loxodromic
elements are the eight values $\pm i,\ \pm1\pm i,\ \pm2i$, and by \lemref{error} (applied
with $T=4$) each of them occurs for only finitely many matrices, hence for only finitely
many conjugacy classes.
The next lemma shows that and $c\in\mathbb{Z}[i]$ appearing as (1,2)-entry
of a matrix in $\mathcal{B}_T$ must have absolute value bounded by $T+4$.

\begin{lemma}\label{error}
Let $(c,d)\in\mathbb{Z}[i]^2$ satisfy
\begin{enumerate}
\item
$ |2d+ck|^2\le T$ and $(2d+ck)^2\neq4,$ 
\item
$d^2+ckd-bc=1 \ \text{ for some } b\in\mathbb{Z}[i]$,
\end{enumerate}
where $k=0,1,i,1+i$. 
Then $|c|\le T+4$. 
\end{lemma}
\begin{proof}

Let $\mu=2d+ck$. Then $\mu^2=4d^2+4dck+c^2k^2$. Using condition (2) above, we get
$$\mu^2-4=c(ck^2+4b).$$
Therefore $c$ divides $\mu^2-4$ in $\mathbb{Z}[i]$; consequently $|c|\le |\mu^2-4|$.
Hence we get $|c|\le T+4$ which completes the proof as $T$ is large.\\
\end{proof}

To deduce the counting reversible classes, we can simply omit the elliptic part since there are finitely many reversible classes in the whole group. 

Note next that all the loxodromic classes with $N\mu\le4$ may also be discarded. The
traces $\mu$ with $N\mu\le4$ that do not lie in $[-2,2]$, and hence do give loxodromic
elements, are $\pm i$, $\pm1\pm i$ and $\pm2i$; for each of them \lemref{error} with
$T=4$ bounds $c$, and then $d=(\mu-ck)/2$ as well, so only finitely many matrices and
therefore only finitely many conjugacy classes arise. We may therefore restrict attention
to loxodromic elements whose trace satisfies $N\mu>4$.

From \thmref{thm6}, we know that an loxodromic element $g\in\mathcal{P}=\mathrm {PSL}(2,\mathbb{Z}[i])$ is \emph{reversible} if and only if $g$ is conjugate to a matrix of the form

\[
M(c,d,k)=
\begin{bmatrix}
d+ck & b \\
c    & d
\end{bmatrix}
\text{  in  } \mathrm{SL}(2,\mathbb{Z}[i])
\]
where, $k\in \{0,1,i,i+1\}.$
Let $\mu=2d+ck$ be the trace. We also know that each conjugacy class of reversible elements contain exactly $8$ elements of this form, it is enough to count such matrices in that specific forms.  We are counting $M(c,d,k)$
such that $|\mu|^2\le T$, $\mu^2\neq 4 $ and $\det(M(c,d,k))=1$.
Observe that unit determinant implies $c\mid\mu^2-4$. Conversely, suppose
$\mu^2-4=cb'$ for some Gaussian integer $b'$,
which gives $4d^2+4dck+c^2k^2-cb'=4$.

Then $\det(M(c,d,k)) =1$ is equivalent to $c^2k^2-cb'=-4bc$. Also, $c\neq0$ (otherwise, $M$ would be a parabolic element), then there is such an element
$b\in\mathbb{Z}[i]$ if and only if
\[
4\mid ck^2-b'.
\]
We see from above that $4\mid c^2k^2-cb'$. 
Therefore it is natural to divide into two separate cases depending on the coprimality of $(1+i)$ and $c$. Note that $1+i$ is the ramified prime above $2$. 

\subsection{$c$ coprime to $1+i$ }

In this case, given $\mu$ and an odd $c$, condition \eqref{eq:b} determines $k$ uniquely
and \eqref{eq:c} is then automatic by \remref{rk:odd}; so the data $(c,d,k)$ with $c$ odd
is determined by the single condition
\[
c\mid\mu^2-4 .
\]
Therefore, we get
\begin{align}
\#\{(c,d,k): \det(M(c,d,k))=1,\ (c,2)=1, \ 4<|\tr(M(c,d,k))|^2\le T\}
&=\sum_{4<|\mu|^2\le T}\ \sum_{\substack{(c,2)=1 \\ c\mid\mu^2-4}} 1 .
\label{eq:count}
\end{align}
The inner sum is a restricted divisor function defined as follows.

\begin{defn}
For $\nu\in\Z[i]$, $\nu\neq0$, let $\tau^*(\nu)$ be the number of ideals of $\Z[i]$
dividing $(\nu)$ that are coprime to $2$; equivalently, the number of divisors of the
largest odd divisor of $\nu$, counted up to units.
\end{defn}

For our counting problem, since $\gcd(\mu-2,\mu+2)\mid4$ and $\tau^*$ counts only
divisors coprime to $2$,
\begin{align}
\sum_{4<|\mu|^2\le T}\ \sum_{\substack{(c,2)=1 \\ c\mid\mu^2-4}} 1
&=\sum_{4<N(\mu)\le T}\tau^*(\mu^2-4)
=\sum_{4<N(\mu)\le T}\tau^*(\mu-2)\,\tau^*(\mu+2).
\label{eq:reduction}
\end{align}
Let 
\[
\B^{(1)}_T:=\{(c,d,k): \det(M(c,d,k))=1,\ (c,1+i)=1, \ 4<|\tr(M(c,d,k))|^2\le T\},
\]
From the restricted divisor sum estimate from Proposition we get.

\begin{prop}\label{propBT1}
As $T\to\infty$,
\begin{align}
|\B^{(1)}_T|=4\,S(T)=\frac{\pi^{3}}{12\,\zeta_{K}(2)}\,T(\log T)^{2}+O\!\left(T\log T\right)
=\frac{\pi}{2G}\,T(\log T)^{2}+O\!\left(T\log T\right) \notag
\end{align}
\end{prop}

\subsection{$1+i$ divides $c$ }
In the previous section the equivalence $\det M(c,d,k)=1\iff c\mid\mu^{2}-4$ was
available only for $c$ coprime to $2$, and the count was correspondingly restricted.
We now remove that restriction. 

Throughout, $\pi=1+i$, $v=v_{\p_2}$ is the valuation at $\p_2=(\pi)$, normalised by
$v(\pi)=1$, so that $v(2)=2$ and $v(4)=4$; concretely $v(z)=v_{2}(N z)$. We write
\[
\mu=2d+ck,\qquad \Delta=\mu^{2}-4,\qquad e=e(\mu):=v(\Delta),\qquad
\lambda=\lambda(\mu):=v(\mu),
\]
and $\Delta=\pi^{e}\Delta_{1}$ with $\Delta_{1}$ odd. Recall $2=-i\pi^{2}$ and
$4=-\pi^{4}$.

The first Lemma gives admissibility criteria in this case.

\begin{lemma}\label{lem:admiss}
Let $\mu\in\Z[i]$ with $\Delta\neq0$. A pair $(c,k)$ with $c\neq0$ and
$k\in\{0,1,i,1+i\}$ occurs as the data of a matrix $M(c,d,k)\in SL(2,\Z[i])$ of trace
$\mu$ if and only if
\begin{align}
&\text{\emph{(a)}}\quad c\mid\Delta,\label{eq:a}\\
&\text{\emph{(b)}}\quad ck\equiv\mu \pmod 2,\label{eq:b}\\
&\text{\emph{(c)}}\quad \Delta/c\equiv ck^{2}\pmod 4,\label{eq:c}
\end{align}
and then $d=(\mu-ck)/2$ and $b=\bigl(\Delta/c-ck^{2}\bigr)/4$ are uniquely determined.
\end{lemma}

\begin{proof}
Given $(c,k)$, the matrix entry $d$ must satisfy $2d=\mu-ck$, which is solvable in
$\Z[i]$ exactly under \eqref{eq:b}. Squaring $\mu=2d+ck$ and using
$\det M=d^{2}+ckd-bc$ gives the identity $\Delta=c\,(ck^{2}+4b)$, so $c\mid\Delta$ and
$b$ exists precisely when $\Delta/c\equiv ck^{2}\pmod 4$. 
\end{proof}

\begin{rk}\label{rk:odd}
Note that \eqref{eq:a} and \eqref{eq:b} imply \eqref{eq:c}
when $v(c)=0$. This matches with the unramified case. Therefore our counting result is obtained by considering this general case. We chose to keep the Case 1 since it has all the analytic ingredients of the proof. 
We shall prove the first case and in the second case, indicate the changes due to presence of the weights in the divisor sums.  
\end{rk}

The next task is to determine possible valuations of $\mu$ and $\Delta$. 

\begin{lemma}\label{lem:e}
With $\lambda=v(\mu)$ and $e=v(\Delta)$,
\begin{align}
e=0\iff\lambda=0,\qquad e=2\iff\lambda=1,\qquad e=4\iff\lambda\ge3,\qquad
e\ge6\iff\lambda=2 .
\end{align}
In the last case, writing $\mu=2\nu$ with $\nu$ a unit and $\nu=1+\pi t$
\textup{(}every unit of $\Z_2[i]$ is $\equiv1\bmod\pi$\textup{)},
\begin{align}
e=6+v(t)+v(t-i\pi).
\label{eq:elam2}
\end{align}
In particular $e$ never takes the values $1,3,5,7,8$: the set of possible valuations is
$\{0,2,4,6\}\cup\{9,10,11,\dots\}$.
\end{lemma}

\begin{proof}
$e=v(\mu-2)+v(\mu+2)$. If $\lambda=0$ both factors are odd and $e=0$. If $\lambda=1$
then $v(\mu\pm2)=\min(1,2)=1$ and $e=2$. If $\lambda\ge2$ write $\mu=2\nu$, so
$\mu\pm2=2(\nu\pm1)$ and $e=4+v(\nu-1)+v(\nu+1)$; for $\nu$ even (i.e. $\lambda\ge3$)
both $\nu\pm1$ are odd and $e=4$. For $\nu$ odd (i.e. $\lambda=2$) put $\nu=1+\pi t$;
then $\nu-1=\pi t$, while $\nu+1=2+\pi t=\pi(t-i\pi)$ because $2=-i\pi^{2}$. This gives
\eqref{eq:elam2}. Finally, since $v(i\pi)=1$:
if $v(t)=0$ then $v(t-i\pi)=0$ and $e=6$; if $v(t)\ge2$ then $v(t-i\pi)=1$ and
$e=7+v(t)\ge9$; and if $v(t)=1$, writing $t=\pi s$ and $s/i=1+\pi r$ we get
$t-i\pi=\pi(s-i)=i\pi^{2}r$, so $v(t-i\pi)=2+v(r)$ and $e=9+v(r)\ge9$.
\end{proof}

Now, we observe that counting $(c,k)$ is the same as counting the triples $(j,\gamma,k)$
where $c=\pi^j\gamma$ with $\gamma$ coprime to $\pi$. 
For fixed $(j,k)$, $\gamma$ runs over divisors of $\Delta_1$.
The number of such $\gamma$ is 
$$\tau(\Delta_1)=\tau^*(\Delta)=\tau^*(\mu-2)\tau^*(\mu+2).$$
Given $\mu$, let $W(e)$ be the number of admissible pairs
$(j,k)$. Then for every $T$,
\begin{align}
\bigl|\B_{T}\bigr|
=4\sum_{\substack{\mu\in\Z[i]\\ 4<N\mu\le T}}
W\bigl(e(\mu)\bigr)\,\tau^{*}(\mu-2)\,\tau^{*}(\mu+2).
\label{eq:identity}
\end{align}

Combining the count \eqref{eq:count} with the divisor-sum asymptotics developed in the
sections that follow, we obtain the main result of this paper: an asymptotic formula,
with explicit leading constant, for the number of special reversible matrices ordered by
the norm of their trace. Writing
\[
\B_T:=\{(c,d,k): \det(M(c,d,k))=1,\ 4<|\tr(M(c,d,k))|^2\le T\},
\]
we prove the following.

\begin{theorem}\label{thm:BT}
As $T\to\infty$,
\begin{align}
\bigl|\B_{T}\bigr|
=\frac{\pi^{3}}{4\,\zeta_{K}(2)}\;T(\log T)^{2}+O\!\left(T\log T(\log\log T)^{2}\right)
=\frac{3\pi}{2G}\;T(\log T)^{2}+O\!\left(T\log T(\log\log T)^{2}\right),
\end{align}
where $G=L(2,\chi_{-4})=0.9159655\ldots$ is Catalan's constant and
$\dfrac{3\pi}{2G}=5.144722\ldots$
\end{theorem}

The proof occupies the rest of the paper.

\subsection{Ingham's additive divisor theorem over $\mathbb{Z}[i]$}
\label{sec:ingham}
We shall use the following analog of the classical
result of Ingham \cite{ingham1927} for the Gaussian integers.

\begin{theorem}\label{thm:ingham}
Fix $\kappa\in\mathbb{Z}[i]\setminus \{0\}$.
Then, as $x\to\infty$,
\begin{align}
\sum_{\substack{\nu\in\mathbb{Z}[i]\\ 0<N\nu\le x}}
\tau(\nu)\,\tau(\nu+\kappa)
&=\frac{\pi^{3}}{16\,\zeta_{K}(2)}\,\sigma_{-1}(\kappa)\,x(\log x)^{2}+O\!\left(x\log x\right),
\end{align}
where $\tau$ is the divisor function on $\mathbb{Z}[i]$,
$\zeta_{K}(s)$ is the Dedekind zeta function of $K=\mathbb{Q}(i)$,
and $\sigma_{-1}(\kappa)=\sum_{\delta\mid\kappa}N(\delta)^{-1}$.
\end{theorem}
In this section we prove Theorem~\ref{thm:ingham}. Throughout, $K=\mathbb{Q}(i)$ and,
for $\nu\neq0$, $\tau(\nu):=\#\{\mathfrak{d}\ \text{ideal of}\ \mathbb{Z}[i]:
\mathfrak{d}\mid(\nu)\}$ is the number of divisors of the ideal $(\nu)$.

\begin{lemma}\label{lem:lattice}
Let $\gamma\in\mathbb{Z}[i]$, $\gamma\neq0$, let $c\in\mathbb{C}$ and $\varrho>0$. Then
$$\#\{h\in\mathbb{Z}[i]:|c+\gamma h|\le\varrho\}=\frac{\pi\varrho^{2}}{N\gamma}+O\!\left(\frac{\varrho}{\sqrt{N\gamma}}+1\right),$$
with an absolute implied constant.
\end{lemma}

Next we record some elementary counting results. The proofs are routine extension of analogous results for integers.
\begin{lemma}\label{lem:counts}
For $y\ge2$:

\begin{enumerate}
\item
$A(y):=\#\{\alpha:N\alpha\le y\}=\pi y+O(\sqrt{y})$
\item
$T(y):=\displaystyle\sum_{N\alpha\le y}\frac{1}{N\alpha}=\pi\log y+O(1)$, and $\displaystyle\sum_{N\alpha\le y}\frac{1}{\sqrt{N\alpha}}\ll\sqrt{y}$
\item
$\#\{(\alpha,\beta):N\alpha N\beta\le y\}\ll y\log y$, and $\displaystyle\sum_{N\alpha N\beta\le y}\frac{1}{\sqrt{N\alpha N\beta}}\ll\sqrt{y}\log y$
\item
$t(y):=\displaystyle\sum_{N\alpha N\beta\le y}\frac{1}{N\alpha N\beta}=\frac{\pi^{2}}{2}(\log y)^{2}+O(\log y)$.
\item
$\tau(y):=\displaystyle\sum'_{N\alpha N\beta\le y}\frac{1}{N\alpha N\beta},
=\frac{t(y)}{\zeta_{K}(2)}+O(\log y)$ \\
where $\Sigma'$ indicates sum over mutually coprime
ideals $\alpha, \beta$.
\end{enumerate}

\end{lemma}

\subsection*{Proof of Theorem~\ref{thm:ingham}}

For $\nu\neq0$ the number of ordered pairs $(\mu,\sigma)$ of nonzero Gaussian integers with $\mu\sigma=\nu$ is $4d(\nu)$
due to presence of $4$ units $\pm 1,\pm i$.
Hence $16\,S(x)$ is equal to the number of quadruples of nonzero Gaussian integers $(\lambda,\mu,\rho,\sigma)$ satisfying
\begin{equation}\label{quadruple}
\lambda\rho-\mu\sigma=\kappa,\qquad N(\mu\sigma)\le x.
\end{equation}

For any quadruple $(\lambda, \rho, \mu, \sigma)$ satisfying (\ref{quadruple})
$$N(\lambda\mu)\cdot N(\rho\sigma)=N(\lambda\rho)\,N(\mu\sigma)=N(\nu+\kappa)\,N\nu\le\big(\sqrt{x}+|\kappa|\big)^{2}x=:X^{2},$$
say, so that at least one of the inequalities 
$N(\lambda\mu)\le X$, $N(\rho\sigma)\le X$ hold.

Let $P$ be the number of quadruples satisfying (\ref{quadruple}) for which the first inequality holds, 
$Q$ the number for which the second holds, and $R$ the number for which both hold. 
The involution 
$$(\lambda,\mu,\rho,\sigma)\mapsto(\rho,\sigma,\lambda,\mu)$$ preserves the conditions (\ref{quadruple}) and interchanges the two inequalities, so $Q=P$ and
$$16\,S=P+Q-R=2P-R.$$
The first sum can be decomposed as
$$P=\sum_{N(\lambda\mu)\le X}N(\lambda,\mu),$$
where $N(\lambda,\mu)$ is the number of solutions in nonzero Gaussian integers $\rho,\sigma$ of the equation
\begin{equation}\label{P-equation}
\lambda\rho-\mu\sigma=\kappa
\end{equation}

for which $0<N\sigma\le x/N\mu$.

Let $(\delta)=(\lambda)+(\mu)$ be the gcd ideal, $\delta$ a fixed generator, and write $\lambda=\delta\lambda'$, $\mu=\delta\mu'$, so that $(\lambda',\mu')=(1)$. If $(\delta)\nmid(\kappa)$, equation (\ref{quadruple}) has no solutions and $N(\lambda,\mu)=0$. If $(\delta)\mid(\kappa)$, then dividing by $\delta$ we can write the general solution as
\begin{equation}\label{parametrization}
\rho=\rho_{0}+\mu'h,\qquad\sigma=\sigma_{0}+\lambda'h,
\end{equation}
where $(\rho_{0},\sigma_{0})$ is a particular solution and $h$ an arbitrary Gaussian integer.

Using Lemma~\ref{lem:lattice} and the above inequality $N\sigma\le x/N\mu$ we get
\begin{equation}\label{N-mu-lambda}
N(\lambda,\mu)
=\#\{h\in\mathbb{Z}[i]: |\sigma_0+\lambda'h|\le \sqrt{x/N\mu}\}
=\frac{\pi x}{N\delta\,N\lambda'N\mu'}+O\!\left(\sqrt{\frac{x}{N\delta\,N\lambda'N\mu'}}+1\right)
\end{equation}
where we have also used $N\mu\,N\lambda'=N\delta\,N\lambda'N\mu'$.

Now we can count $P$ as
\[
P=\pi x\sum_{(\delta)\mid(\kappa)}\frac{1}{N\delta}
\sideset{}{'}
\sum_{N(\lambda'\mu')\le X/N\delta^{2}}\frac{1}{N\lambda'N\mu'}\;+\; E,
\]
where, by Lemma~\ref{lem:counts}(3),
\begin{align*}
E\ll & \sum_{(\delta)\mid(\kappa)}\left[\sqrt{\frac{x}{N\delta}}\sum_{N(\lambda'\mu')\le X/N\delta^{2}}\frac{1}{\sqrt{N\lambda'N\mu'}}\;+\;\#\Big\{(\lambda',\mu'):N(\lambda'\mu')\le\frac{X}{N\delta^{2}}\Big\}\right]\\
\ll & \sqrt{xX}\log X+X\log X\ll x\log x.
\end{align*}
Hence, by Lemma~\ref{lem:counts}(5),
$$P=\pi x\sum_{(\delta)\mid(\kappa)}\frac{1}{N\delta}\left\{\frac{\pi^2}{2\zeta_K(2)}\left(\log\frac{X}{N\delta^{2}}\right)^{2}+O(\log X)\right\}+O(x\log x).$$
Expanding $\big(\log X-2\log N\delta\big)^{2}=(\log X)^{2}+O(\log X)$ and using (3),
$$P=\frac{3\pi}{G}\,x(\log X)^{2}\sum_{(\delta)\mid(\kappa)}\frac{1}{N\delta}+O(x\log x)
=\frac{3\pi}{G}\,\sigma_{-1}(\kappa)\,x(\log x)^{2}+O(x\log x)
$$
where $G=\frac{6}{\pi^2}\zeta_K(2)$.

Now we turn to estimating $R$.
We have
$$R\le\sum_{N(\lambda\mu)\le X}N'(\lambda,\mu),$$
where $N'(\lambda,\mu)$ is the number of solutions of 
(\ref{P-equation}) in nonzero $\rho,\sigma$ for which $N(\rho\sigma)\le X$. Again $N'=0$ unless $(\delta)\mid(\kappa)$, in which case we use the parametrization (\ref{parametrization}). Writing $a=-\rho_{0}/\mu'$ and $b=-\sigma_{0}/\lambda'$ for the roots of the quadratic $h\mapsto(\rho_{0}+\mu'h)(\sigma_{0}+\lambda'h)$, the condition $N(\rho\sigma)\le X$ reads
$$|h-a|\,|h-b|\le\frac{\sqrt{X}}{|\lambda'\mu'|}=:T.$$
If both $|h-a|>\sqrt{T}$ and $|h-b|>\sqrt{T}$, the left side exceeds $T$; hence every admissible $h$ lies in one of two discs of radius $\sqrt{T}$, and since a disc of radius $r$ contains $O(r^{2}+1)$ Gaussian integers,
$$N'(\lambda,\mu)\ll T+1=\sqrt{\frac{X}{N\lambda'N\mu'}}+1.$$
Similarly as in the estimate the error term $E$ of $P$,
$$R\ll\sum_{(\delta)\mid(\kappa)}\left(\sqrt{X}\cdot\frac{\sqrt{X}}{N\delta}\log X+\frac{X}{N\delta^{2}}\log X\right)\ll X\log X\ll x\log x.
$$

Combining the estimates of $P$ and $R$.
$$16\,S(x)=\frac{6\pi}{G}\,\sigma_{-1}(\kappa)\,x(\log x)^{2}+O(x\log x),$$
and therefore
$$S(x)=\frac{3\pi}{8G}\,\sigma_{-1}(\kappa)\,x(\log x)^{2}+O(x\log x)=\frac{\pi^{3}}{16\,\zeta_{K}(2)}\,\sigma_{-1}(\kappa)\,x(\log x)^{2}+O(x\log x).$$ \qed

\subsection{Applying the theorem to restricted divisors}
\label{sec:applying}

Theorem~\ref{thm:ingham} is stated for the \emph{full} divisor function $\tau$,
whereas \eqref{eq:reduction} involves the $2$-restricted function $\tau^*$. We now
show that the passage from $\tau$ to $\tau^*$ is effected by a single Euler factor at
the ramified prime, and we track its effect on the main term explicitly. Throughout,
$\p_2=(1+i)$ denotes the unique prime ideal of $\Z[i]$ above $2$, with $N\p_2=2$ and
$(2)=\p_2^2$; ``odd'' means coprime to $\p_2$.

\begin{lemma}\label{lem:separate}
For every $\nu\in\Z[i]\setminus\{0\}$ write $(\nu)=\p_2^{\,a}\,\mathfrak{m}$ with
$\mathfrak{m}$ odd. Then
\[
\tau(\nu)=(a+1)\,\tau^*(\nu),\qquad\text{where }\ \tau^*(\nu)=\tau(\mathfrak m).
\]
\end{lemma}
The following proposition adapts Ingham's theorem to the present
context.

\begin{prop}\label{prop:main}
As $T\to\infty$,
\begin{align}
S(T)=\sum_{N(\mu)\le T}\tau^*(\mu-2)\,\tau^*(\mu+2)
&=\frac{\pi^{3}}{48\,\zeta_{K}(2)}\,T(\log T)^{2}+O\!\left(T\log T\right) \notag\\
&=\frac{\pi}{8G}\,T(\log T)^{2}+O\!\left(T\log T\right),
\end{align}
where $G=L(2,\chi_{-4})=0.9159655\ldots$ is Catalan's constant.
\end{prop}

\begin{proof}
We address the changes at the three places which contribute to the constant of Theorem~\ref{thm:ingham} when $\tau$ is replaced by $\tau^*$ and the
shift is $\kappa=4=(1+i)^4$ (a power of $\p_2$). Recall
$\rho:=\operatorname*{Res}_{s=1}\zeta_K(s)=\pi/4$.\\

\smallskip
\emph{(i) Change in $\sigma_{-1}(\kappa)$:} Expanding
$\tau^*(\mu-2)\,\tau^*(\mu+2)=\sum_{\dd_1\mid(\mu-2)}\sum_{\dd_2\mid(\mu+2)}1$ over
\emph{odd} ideals $\dd_1,\dd_2$ and interchanging, the inner count requires
\begin{align}
\mu\equiv 2 \ (\mathrm{mod}\ \dd_1),\qquad \mu\equiv -2\ (\mathrm{mod}\ \dd_2).
\end{align}
By the Chinese Remainder Theorem these are simultaneously solvable if and only if
$\mathfrak g:=(\dd_1,\dd_2)$ divides $(4)=\p_2^{4}$. But $\mathfrak g$ is odd, so only \emph{coprime} pairs $(\dd_1,\dd_2)$ contribute. Thus  $\sigma_{-1}(\kappa)$, which in the unrestricted problem records the
common divisors of $\dd_1,\dd_2$ dividing $\kappa$, is now replaced by $1$.

\smallskip
\emph{(ii) Change in Dirichlet series:}
The relevant Dirichlet
series is $(1-2^{-s})\zeta_K(s)$ rather than $\zeta_K(s)$, and the residue at $s=1$ becomes
\begin{align}
\rho^*:=\operatorname*{Res}_{s=1}\big[(1-2^{-s})\zeta_K(s)\big]
=\tfrac12\rho=\frac{\pi}{8},
\end{align}
This gives 
$$\sum_{\mathfrak e\ \mathrm{odd},\,N\mathfrak e\le Y}N\mathfrak e^{-1}
=\tfrac{\pi}{8}\log Y+O(1).$$ 
As this occurs once for each of the two factors, so the constant is multiplied by $(\rho^*/\rho)^2=\tfrac14$.

\smallskip
\emph{(iii) Change in the Zeta-factor:} 
Removing the coprimality condition
$(\dd_1,\dd_2)=(1)$ by M\"obius inversion over \emph{odd} ideals replaces
$\frac{1}{\zeta_K(2)}=\prod_{\p}\big(1-N\p^{-2}\big)$ by
\begin{align}
\prod_{\p\ \mathrm{odd}}\big(1-N\p^{-2}\big)
=\frac{1}{\zeta_K(2)}\Big(1-N\p_2^{-2}\Big)^{-1}
=\frac{1}{\zeta_K(2)}\Big(1-\tfrac14\Big)^{-1}
=\frac{4}{3\,\zeta_K(2)} .
\end{align}

\smallskip
Starting from the theorem's constant
$\frac{\pi^3}{16\zeta_K(2)}\sigma_{-1}(\kappa)$, 
applying (i),(ii), and (iii), the new constant is
\begin{align*}
C^*=\frac{\pi^{3}}{16\,\zeta_K(2)}\cdot\underbrace{1}_{(i)}\cdot
\underbrace{\tfrac14}_{(ii)}\cdot\underbrace{\tfrac43}_{(iii)}
=\frac{\pi^{3}}{16\,\zeta_K(2)}\cdot\frac13
=\frac{\pi^{3}}{48\,\zeta_K(2)} .
\end{align*}
Recall
$\zeta_K(s)=\zeta(s)L(s,\chi_{-4})$, so
$\zeta_K(2)=\zeta(2)L(2,\chi_{-4})=\frac{\pi^2}{6}G$, where $G$
is the Catalan constant. So the constant in the main term is
\begin{align*}
\frac{\pi^{3}}{48\,\zeta_K(2)}=\frac{\pi}{8G}.
\end{align*}
\end{proof}

\begin{rk}
The error term $O(T\log T)$ is the limit of Ingham's elementary method. Improving it to a
power saving $O(T^{1-\delta})$ requires cancellation in sums of Kloosterman sums for the
Picard group $PSL(2,\Z[i])$, exactly as the Deshouillers--Iwaniec bound
$O(x^{2/3+\varepsilon})$ replaces Ingham's $O(x\log x)$ in the classical rational case. Furthermore we refer to \cite{bruggeman2003sum} for Kloosterman sums estimates over the Gaussian number field.
\end{rk}

\subsection{The ramified part: matrices with $(1+i)\mid c$}
\label{sec:ramified}

\subsubsection{The local weights $W(e)$ :}

Fix $\mu$ and $j$, and write $c=\pi^{j}\gamma$, $\Delta/c=\pi^{e-j}\sigma$ with
$\gamma,\sigma$ odd and $\gamma\sigma=\Delta_{1}$. 
By Lemma~\ref{lem:admiss} we have $\Delta-c^{2}k^{2}=4bc$, so that condition
\eqref{eq:c} may be restated as
\begin{equation}\label{eq:cprime}
\eqref{eq:c}\ \text{ holds if and only if }\ 
v\!\left(\frac{\Delta}{c}-ck^{2}\right)\ge 4.
\end{equation}
This can be rewritten depending on values of $k$ as
\begin{align}
k=0:\ &\ e-j\ge4; \notag\\
k\in\{1,i\}:\ &\ v\bigl(\pi^{e-j}\sigma\mp\pi^{j}\gamma\bigr)\ge4; \label{eq:cases}\\
k=1+i:\ &\ v\bigl(\pi^{e-j}\sigma-\pi^{j+2}\gamma\bigr)\ge4. \notag
\end{align}
Moreover, \eqref{eq:b} reads $j+v(k)=\lambda$ if $\lambda\le1$, and $j+v(k)\ge2$ if
$\lambda\ge2$.

The next lemma specifies the values taken by $W(e)$

\begin{prop}\label{prop:W}
For each $j$ the number of admissible $k$ does not depend on which $c$ with $v(c)=j$
is chosen. Writing $W(e)$ for the total number of admissible pairs $(j,k)$,
\begin{align}
W(0)=1,\qquad W(2)=2,\qquad W(4)=4,\qquad W(6)=6,\qquad
W(e)=4e-22\quad(e\ge9).
\label{eq:Wvalues}
\end{align}
\end{prop}

\begin{proof}
\emph{Case $e=0$.} Here $j=0$ and \eqref{eq:b} determines $k$ uniquely among the two
unit classes $\{1,i\}$; \eqref{eq:c} is automatic. Hence $W(0)=1$.

\emph{Case $e=2$ \textup{(}$\lambda=1$\textup{)}.} Condition \eqref{eq:b} forces
$j+v(k)=1$. For $j=0$ we must take $k=1+i$, and \eqref{eq:c} is automatic. For $j=1$
we have $k\in\{1,i\}$ and \eqref{eq:cases} becomes
$v(\sigma\mp\gamma)\ge3$, i.e. $\Delta_{1}\equiv\pm\gamma^{2}\pmod{\pi^{3}}$ after
multiplying by the unit $\gamma$. The squares of units modulo $\pi^{3}$ are
$\{1,-1\}$, and $\{\gamma^{2},-\gamma^{2}\}$ is exactly this set; moreover
$\gamma^{2}\not\equiv-\gamma^{2}$ since $v(2\gamma^{2})=2<3$. Hence \emph{exactly one}
sign can occur, and it does occur precisely when $\Delta_{1}$ is a square mod
$\pi^{3}$. Writing $\mu\mp2=\pi\alpha_{\mp}$ we have
$\alpha_{+}-\alpha_{-}=4/\pi=-\pi^{3}$, so $\Delta_{1}=\alpha_{-}\alpha_{+}\equiv
\alpha_{-}^{2}$ is a square. Thus $j=1$ contributes exactly $1$, and $W(2)=2$.

\emph{Case $e=4$ \textup{(}$\lambda\ge3$\textup{)}.} Now \eqref{eq:b} is
$j+v(k)\ge2$. For $j=0$ only $k=0$ qualifies, and \eqref{eq:c} is automatic. For $j=1$
we may take $k\in\{0,1+i\}$: the choice $k=0$ fails since $e-j=3<4$, while $k=1+i$
gives $\pi^{3}(\sigma-\gamma)$, of valuation $\ge3+1=4$, so it succeeds. For $j=2$ all
four $k$ pass \eqref{eq:b}; $k=0$ fails ($e-j=2$), $k=1+i$ fails
($v(\pi^{2}\sigma-\pi^{4}\gamma)=2$), while for $k\in\{1,i\}$ we need
$v(\sigma\mp\gamma)\ge2$, and since $-1\equiv1\pmod{\pi^{2}}$ both conditions coincide
and read $\Delta_{1}\equiv\gamma^{2}\equiv1\pmod{\pi^{2}}$. Here
$\Delta=-\pi^{4}(\nu^{2}-1)$ with $\nu$ even, so $\Delta_{1}=1-\nu^{2}\equiv1$; both
$k=1$ and $k=i$ succeed. For $j\in\{3,4\}$ we have $v(\Delta/c)=e-j\le1$ while
$v(ck^{2})\ge3$, so \eqref{eq:cprime} fails. Hence $W(4)=1+1+2=4$.

\emph{Case $e\ge6$ \textup{(}$\lambda=2$\textup{)}.} Again \eqref{eq:b} is
$j+v(k)\ge2$.
\begin{itemize}
\item $k=0$ is admissible for every $j$, and \eqref{eq:cases} holds iff $j\le e-4$:
this gives $e-3$ pairs.
\item $k\in\{1,i\}$ requires $j\ge2$. If $j\neq e-j$ the valuation in
\eqref{eq:cases} equals $\min(j,e-j)$, so the condition is $4\le j\le e-4$, and then
\emph{both} signs work: $2\max(0,e-7)$ pairs. If $j=e/2$ the expression is
$\pi^{j}(\sigma\mp\gamma)$ of valuation $\ge j+1$, so both signs work as soon as
$j\ge3$; for $e\ge8$ even this $j$ already lies in $[4,e-4]$, whereas for $e=6$ it is
the extra value $j=3$, contributing $2$ more pairs.
\item $k=1+i$ requires $j\ge1$. If $e-j\neq j+2$ the condition is
$\min(e-j,j+2)\ge4$, i.e. $2\le j\le e-4$. If $e-j=j+2$ the expression is
$\pi^{j+2}(\sigma-\gamma)$, of valuation $\ge j+3\ge4$, so it succeeds; and
$j=(e-2)/2$ lies in $[2,e-4]$ for $e\ge6$. Altogether $\max(0,e-5)$ pairs.
\end{itemize}
Summing,
\begin{align}
W(e)=(e-3)+2\max(0,e-7)+2\,[e=6]+\max(0,e-5),
\end{align}
which equals $6$ for $e=6$ and $4e-22$ for $e\ge9$. In every case the criterion
depended on $\gamma$ only through quantities ($\Delta_{1}$ modulo a power of $\pi$, or
nothing at all) that are independent of the divisor chosen, which proves the first
assertion.
\end{proof}

\subsection{Averaging the local weight}

The weight $W(e(\mu))$ depends on $\mu$ only through its image in $\Z_{2}[i]$, the ring
of $2$-adic Gaussian integers (the completion of $\mathbb Z[i]$ at $\mathfrak{p}$); indeed by
Lemma~\ref{lem:e} the value $e(\mu)$ is determined by the class of $\mu$ modulo a power
of $\pi$. Equip $\Z_{2}[i]$ with its normalised additive Haar measure $\Pb$, under which
each residue class modulo $\pi^{n}$ has measure $2^{-n}=N(\pi^{n})^{-1}$. This is the
natural measure to average against, because the Gaussian integers $\mu$ with $N\mu\le T$ become equidistributed among residue classes modulo any fixed $\pi^{n}$ (see Lemma~\ref{lem:equi} below). 

The first lemma of this section is on the explicit distribution function of $\lambda=v(\mu)$, now considered a random variable on $\Z_{2}[i]$.
\begin{lemma}\label{lem:dist}
Under the Haar measure $\Pb$ on $\Z_2[i]$:
\begin{enumerate}
\item[\textup{(i)}] $\Pb\bigl(v(\mu)\ge n\bigr)=2^{-n}$ for all $n\ge0$; equivalently
$\Pb\bigl(v(\mu)=n\bigr)=2^{-n-1}$. In particular
\begin{align}
\Pb(\lambda=0)=\tfrac12,\quad \Pb(\lambda=1)=\tfrac14,\quad
\Pb(\lambda=2)=\tfrac18,\quad \Pb(\lambda\ge3)=\tfrac18 .
\end{align}
\item[\textup{(ii)}] Conditional on $\lambda=2$, write $\mu=2\nu$ with $\nu=1+\pi t$ as
in Lemma~\ref{lem:e}; then $\Pb(v(t)=m)=2^{-m-1}$.
Moreover, given $v(t)=1$, $\Pb(v(r)=n)=2^{-n-1}$ (where $r$ is as in the proof of Lemma \ref{lem:e}.
\end{enumerate}
\end{lemma}

\begin{proof}
(i) The event $v(\mu)\ge n$ is $\mu\in\pi^{n}\Z_{2}[i]$, a single residue class modulo
$\pi^{n}$, of Haar measure $2^{-n}$. Clearly
\[
\Pb(v(\mu)=n)=\Pb(v(\mu)\ge n)-\Pb(v(\mu)\ge n+1)=2^{-n}-2^{-n-1}=2^{-n-1}.
\] 
The four displayed probabilities are the cases $n=0,1,2$ and the tail $n\ge3$.\\ 
(ii) The map $\mu\mapsto t=(\mu/2-1)/\pi$ is, on the
class $\lambda=2$, an affine bijection onto $\Z_{2}[i]$ carrying normalised Haar measure
to normalised Haar measure. Hence (i) applied to $t$ (and then to $r$)
gives the stated laws.
\end{proof}

\begin{prop}\label{prop:EW}
With $\Pb$ as above, the average of the local weight is
$\displaystyle \mathbb{E}_{\Pb}\bigl[W(e(\mu))\bigr]=3 .$
\end{prop}

\begin{proof}
By Lemma~\ref{lem:e} the strata $\lambda=0,1,\ge3$ give $e=0,2,4$ respectively, with
weights $W(0)=1$, $W(2)=2$, $W(4)=4$ and probabilities $\tfrac12,\tfrac14,\tfrac18$ from
Lemma~\ref{lem:dist}(i). It remains to compute $\mathbb E_\Pb[W\mid\lambda=2]$. By
Lemma~\ref{lem:dist}(ii) and the proof of Lemma~\ref{lem:e},
\begin{align}
e=\begin{cases}
6, & v(t)=0 \quad (\text{prob. } \tfrac12),\\[2pt]
7+m, & v(t)=m\ge2 \quad (\text{prob. } 2^{-m-1}),\\[2pt]
9+n, & v(t)=1,\ v(r)=n \quad (\text{prob. } \tfrac14\cdot2^{-n-1}).
\end{cases}
\end{align}
Using \eqref{eq:Wvalues}, so that $W(6)=6$ and $W(e)=4e-22$ for $e\ge9$ (note $W(7+m)$
is used only for $m\ge2$, i.e. $e\ge9$, and $W(9+n)$ for $n\ge0$), together with
$\sum_{m\ge0}2^{-m-1}=1$ and $\sum_{m\ge0}m\,2^{-m-1}=1$,
\begin{align}
\mathbb{E}_\Pb\bigl[W\mid\lambda=2\bigr]
&=\underbrace{\tfrac12\cdot 6}_{v(t)=0}
+\underbrace{\sum_{m\ge2}2^{-m-1}\,W(7+m)}_{v(t)\ge2}
+\underbrace{\tfrac14\sum_{n\ge0}2^{-n-1}\,W(9+n)}_{v(t)=1}\notag\\
&=3+\sum_{m\ge2}2^{-m-1}(4m+6)+\tfrac14\sum_{n\ge0}2^{-n-1}(4n+14).
\end{align}
For the middle sum, $\sum_{m\ge2}2^{-m-1}=\tfrac14$ and
$\sum_{m\ge2}m\,2^{-m-1}=1-\tfrac12\cdot\tfrac12=\tfrac34$, so it equals
$4\cdot\tfrac34+6\cdot\tfrac14=\tfrac92$. For the last sum,
$\tfrac14\bigl(4\cdot1+14\cdot1\bigr)=\tfrac14\cdot18=\tfrac92$. Hence
$\mathbb E_\Pb[W\mid\lambda=2]=3+\tfrac92+\tfrac92=12$, and therefore
\begin{align}
\mathbb{E}_\Pb[W]=\tfrac12\cdot1+\tfrac14\cdot2+\tfrac18\cdot4+\tfrac18\cdot12
=\tfrac12+\tfrac12+\tfrac12+\tfrac32=3.
\end{align}
\end{proof}

Thus the average weight over the ramified place is $3$ times the weight $W(0)=1$ carried by the odd-$c$ matrices.

The following evaluates the additive restricted divisor sums over arithmetic progressions. This will be used in the next section to remove weights from divisor sums. 

\begin{lemma}\label{lem:equi}
There is an absolute constant $C_0$ such that, for every $n\ge1$ and every $a\in\Z[i]$
subject to
\begin{equation}\label{eq:equihyp}
v(a-2)<n\qquad\text{and}\qquad v(a+2)<n,
\end{equation}
\begin{align}
\sum_{\substack{N\mu\le T\\ \mu\equiv a\ (\mathrm{mod}\ \pi^{n})}}
\tau^{*}(\mu-2)\tau^{*}(\mu+2)
=\frac{1}{2^{n}}\cdot\frac{\pi}{8G}\,T(\log T)^{2}+O\!\left(C_0\,T\log T\right),
\label{eq:equi}
\end{align}
where the implied constant is \emph{absolute}: it does not depend on $n$ or on $a$.
\end{lemma}

\begin{proof}
The argument proving Theorem~\ref{thm:ingham} can be easily adapted to this case. In fact, inserting the extra congruence $\mu\equiv a\ (\mathrm{mod}\ \pi^{n})$, opening both $\tau^{*}$'s over odd ideals
$\dd_{1},\dd_{2}$, the inner quantity is the number of $\mu$ in the disc $N\mu\le T$
lying in a single prescribed class modulo $\mathfrak q:=\dd_{1}\dd_{2}\pi^{n}$ by the Chinese Remainder Theorem. By Lemma~\ref{lem:lattice} (lattice points in a disc), this count is
\begin{align}
\frac{\pi T}{2^{n}N\dd_1N\dd_2}
+O\!\left(\frac{\sqrt{T}}{\sqrt{2^{n}N\dd_1N\dd_2}}+1\right).
\label{eq:innercount}
\end{align}
The main term is exactly $2^{-n}$ times that of the unrestricted problem, and summing it
over $\dd_1,\dd_2$ reproduces $2^{-n}\cdot\frac{\pi}{8G}T(\log T)^2$ as in the proof of
Theorem~\ref{thm:ingham}. It remains to bound the two error contributions in
\eqref{eq:innercount} summed over the admissible range $N\dd_1,N\dd_2\le\sqrt T$.
 
The odd ideals $\dd_1,\dd_2$ are unconstrained by the congruence
$\mu\equiv a\ (\mathrm{mod}\ \pi^{n})$, since they are coprime to $\pi$.
Hence they range over exactly the same set as in the unrestricted problem. 
Their number is
$\ll\#\{(\dd_1,\dd_2):N\dd_1N\dd_2\le T\}\ll T\log T$ by Lemma~\ref{lem:counts}(3).
Therefore the sum over the last error term $O(1)$ contributes $O(T\log T)$ with an absolute constant
independent of $n$ .

For the first $O$-term, we observe
\begin{align}
\frac{\sqrt T}{2^{n/2}}\sum_{N\dd_1,N\dd_2\le\sqrt T}\frac{1}{\sqrt{N\dd_1N\dd_2}}
\ll\frac{\sqrt T}{2^{n/2}}\cdot\sqrt T\log T=\frac{1}{2^{n/2}}\,T\log T
\ll T\log T.
\end{align}

Both error contributions are thus $O(T\log T)$ with an implied constant independent of
$n$ and $a$, proving \eqref{eq:equi}.
\end{proof}

\begin{rk}\label{rk:equihyp}
Hypothesis \eqref{eq:equihyp} cannot be omitted, without it the conclusion is false; for instance for $a=2$ and
$n\ge5$ the left-hand side of \eqref{eq:equi} equals
$\sum_{N\gamma\le T/2^{n}+O(\sqrt T)}\tau^{*}(\gamma)\,\tau(1-\pi^{n-4}\gamma)$, whose main
term is smaller by a factor $\tfrac12$ because
$\sum_{\mathfrak n}\tau^{*}(\mathfrak n)N\mathfrak n^{-s}=(1-2^{-s})\zeta_K(s)^{2}$ while
$\sum_{\mathfrak n}\tau(\mathfrak n)N\mathfrak n^{-s}=\zeta_K(s)^{2}$.
\end{rk}

\medskip
We shall see later that the terms of the divisir sum with $v(\mu)=2$ requires careful treatment. We introduce the following parametrisation. If $v(\mu)=2$, then $\mu=2\nu$ with $\nu$ odd, hence
$\nu\equiv1\pmod\pi$, and $\mu$ determines a unique $t\in\Z[i]$ with
$\nu=1+\pi t$; conversely every $t\in\Z[i]$ arises in this way. Since
$2=-i\pi^{2}$ we have $\mu-2=2\pi t$ and $\mu+2=2(2+\pi t)=2\pi(t-i\pi)$, so
that
\begin{align}
\tau^{*}(\mu-2)=\tau^{*}(t),\qquad
\tau^{*}(\mu+2)=\tau^{*}(t-i\pi),\qquad
e(\mu)=6+v(t)+v(t-i\pi),
\label{eq:param}
\end{align}
and $N\mu\le T$ forces $Nt\le T$ for all $T$ large enough.

\begin{lemma}\label{lem:levels}
Let $e\ge6$. The level set $L_{e}:=\{\mu\in\Z[i]:v(\mu)=2,\ e(\mu)=e\}$ is empty
unless $e=6$ or $e\ge9$; it is a single residue class modulo $\pi^{4}$ when
$e=6$, and a union of exactly two residue classes modulo $\pi^{e-3}$ when
$e\ge9$. Consequently
\begin{align}
q_{e}:=\Pb(L_{e})=
\begin{cases}
2^{-4}, & e=6,\\
2^{4-e}, & e\ge9,
\end{cases}
\label{eq:qe}
\end{align}
and, with $W$ as in \eqref{eq:Wvalues},
\begin{align}
\sum_{e\ge6}W(e)\,q_{e}=\frac32,
\qquad
\sum_{e>E}W(e)\,q_{e}\ll E\,2^{-E}\quad(E\ge9).
\label{eq:series}
\end{align}
\end{lemma}

\begin{proof}
By \eqref{eq:param} and the proof of Lemma~\ref{lem:e}, we get: $v(t)=0$ implies $e=6$; $v(t)=m\ge2$ implies $e=7+m$; and
$v(t)=1$, in which case $t=i\pi(1+\pi r)$ for a unique $r\in\Z[i]$, gives
$e=9+v(r)$. Each of the conditions $v(t)=0$, $v(t)=m$ and
$\bigl(v(t)=1,\ v(r)=n\bigr)$ describes a single residue class of $t$ modulo
$\pi$, $\pi^{m+1}$ and $\pi^{n+3}$ respectively. Since $\mu=2+2\pi t=2-i\pi^{3}t$,
a residue class of $t$ modulo $\pi^{k}$ is a residue class of $\mu$ modulo
$\pi^{k+3}$; this yields the stated moduli $\pi^{4}$ for $e=6$ and $\pi^{e-3}$
for $e=7+m\ge9$ and for $e=9+n$, the two branches contributing one class each
when $e\ge9$. Taking Haar measures gives \eqref{eq:qe}.

For \eqref{eq:series}, write $e=9+n$ and use $W(6)=6$, $W(e)=4e-22$ for $e\ge9$:
\[
\sum_{e\ge6}W(e)q_{e}
=\frac{6}{16}+\frac{1}{32}\sum_{n\ge0}(4n+14)2^{-n}
=\frac38+\frac{36}{32}=\frac32 ,
\]
and the tail estimate is immediate from $W(e)q_{e}\ll e\,2^{-e}$.
\end{proof}

The next two lemmas supply the uniform upper bound needed to bound the tail of the sum over $e$.

\begin{lemma}\label{lem:tausq}
For $x\ge2$,
\[
\sum_{0<N\nu\le x}\tau(\nu)^{2}\ \ll\ x(\log x)^{3},
\]
the implied constant being absolute.
\end{lemma}

\begin{proof}
Both $\tau^{2}$ and the fourfold divisor function $\tau_{4}$ are multiplicative,
and on prime powers $(a+1)^{2}\le\binom{a+3}{3}$ for every $a\ge0$; hence
$\tau^{2}\le\tau_{4}$. Writing $\tau_{4}=\tau_{3}*\mathbf 1$ and iterating
Lemma~\ref{lem:counts}(1),(2) gives $\sum_{N\nu\le x}\tau_{4}(\nu)\ll x(\log
x)^{3}$.
\end{proof}

\begin{lemma}\label{lem:tail}
For every integer $E\ge9$ and every $T\ge3$,
\begin{align}
0\ \le\!\!\sum_{\substack{N\mu\le T,\ v(\mu)=2\\ e(\mu)>E}}\!\!
W\bigl(e(\mu)\bigr)\,\tau^{*}(\mu-2)\,\tau^{*}(\mu+2)
\ \ll\ E\,2^{-E/2}\,T(\log T)^{3},
\label{eq:tail}
\end{align}
with an absolute implied constant.
\end{lemma}

\begin{proof}
By \eqref{eq:param} and the proof of Lemma~\ref{lem:levels}, the terms with
$e(\mu)>E$ split into the family $v(t)=m$ with $m\ge E-6$, on which $e=7+m$, and
the family $v(t)=1$, $v(r)=n$ with $n\ge E-8$, on which $e=9+n$.

Consider the first family and fix $m\ge E-6\ge3$. Write $t=\pi^{m}u$ with $u$
odd, so that $\tau^{*}(t)=\tau(u)$; since $m\ge2$ the element
$\pi^{m-1}u-i$ is odd and $t-i\pi=\pi(\pi^{m-1}u-i)$, whence
$\tau^{*}(t-i\pi)=\tau(\pi^{m-1}u-i)$. The constraint $Nt\le T$ reads
$Nu\le T2^{-m}$. The map $u\mapsto\pi^{m-1}u-i$ is injective and
$N(\pi^{m-1}u-i)\ll T$ on this range, so Cauchy--Schwarz and
Lemma~\ref{lem:tausq} give
\[
\sum_{\substack{Nt\le T\\ v(t)=m}}\tau^{*}(t)\,\tau^{*}(t-i\pi)
\le\Bigl(\sum_{Nu\le T2^{-m}}\!\!\tau(u)^{2}\Bigr)^{1/2}
\Bigl(\sum_{Nu\le T2^{-m}}\!\!\tau(\pi^{m-1}u-i)^{2}\Bigr)^{1/2}
\ll\frac{T}{2^{m/2}}(\log T)^{3}.
\]
Since $W(7+m)=4m+6$, summing over $m\ge E-6$ contributes
$\ll E\,2^{-E/2}T(\log T)^{3}$.

For the second family fix $n\ge E-8$ and write $t=i\pi(1+\pi r)$ with
$r=\pi^{n}w$, $w$ odd. Then $\tau^{*}(t)=\tau(1+\pi^{n+1}w)$ and
$t-i\pi=i\pi^{2}r$, so $\tau^{*}(t-i\pi)=\tau(w)$, while $Nt\le T$ forces
$Nw\ll T2^{-n}$. The same application of Cauchy--Schwarz and
Lemma~\ref{lem:tausq}, now with the injection $w\mapsto1+\pi^{n+1}w$, bounds the
inner sum by $\ll T2^{-n/2}(\log T)^{3}$; as $W(9+n)=4n+14$, summing over
$n\ge E-8$ again contributes $\ll E\,2^{-E/2}T(\log T)^{3}$.
\end{proof}

\subsection*{Proof of Theorem \ref{thm:BT}}
By the identity \eqref{eq:identity},
\begin{align}
\bigl|\B_{T}\bigr|
=4\sum_{\substack{4<N\mu\le T}}W\bigl(e(\mu)\bigr)\,
\tau^{*}(\mu-2)\,\tau^{*}(\mu+2),
\end{align}
and we evaluate the right-hand side by splitting the range of $\mu$ according to
$\lambda=v(\mu)$:
\begin{align}
\sum_{\mu}=\sum_{v(\mu)=0}+\sum_{v(\mu)=1}+\sum_{v(\mu)\ge3}+\sum_{v(\mu)=2}.
\label{eq:split4}
\end{align}
By Lemma~\ref{lem:e} the weight $W(e(\mu))$ is constant on the first three
ranges, equal to $W(0)=1$, $W(2)=2$ and $W(4)=4$ respectively; only the range
$v(\mu)=2$ carries a weight that varies, over infinitely many levels.

\medskip
\emph{Step 1: the ranges of constant weight.}
The sets $\{v(\mu)=0\}$, $\{v(\mu)=1\}$ and $\{v(\mu)\ge3\}$ are unions of
residue classes modulo $\pi$, $\pi^{2}$ and $\pi^{3}$, of $\Pb$-measure
$\tfrac12$, $\tfrac14$ and $\tfrac18$ by Lemma~\ref{lem:dist}(i). Applying
Lemma~\ref{lem:equi} to each class and summing over the classes in each range,
\begin{align}
\sum_{v(\mu)=0}\!\!\tau^{*}(\mu-2)\tau^{*}(\mu+2)
&=\tfrac12\cdot\tfrac{\pi}{8G}T(\log T)^{2}+O(T\log T),\\
\sum_{v(\mu)=1}\!\!\tau^{*}(\mu-2)\tau^{*}(\mu+2)
&=\tfrac14\cdot\tfrac{\pi}{8G}T(\log T)^{2}+O(T\log T),\\
\sum_{v(\mu)\ge3}\!\!\tau^{*}(\mu-2)\tau^{*}(\mu+2)
&=\tfrac18\cdot\tfrac{\pi}{8G}T(\log T)^{2}+O(T\log T).
\end{align}
Weighting by $W(0)$, $W(2)$, $W(4)$, these three ranges contribute to
$\tfrac14|\B_{T}|$ the amount
\begin{align}
\Bigl(1\cdot\tfrac12+2\cdot\tfrac14+4\cdot\tfrac18\Bigr)\frac{\pi}{8G}T(\log T)^{2}
+O(T\log T)
=\tfrac32\cdot\frac{\pi}{8G}T(\log T)^{2}+O(T\log T).
\label{eq:step1}
\end{align}

\medskip
\emph{Step 2: the range $v(\mu)=2$.}
Put
\begin{align}
\Sigma_{2}(T):=\sum_{v(\mu)=2}W\bigl(e(\mu)\bigr)\,\tau^{*}(\mu-2)\tau^{*}(\mu+2)
=\sum_{e\ge6}W(e)\!\!\sum_{\substack{N\mu\le T\\ \mu\in L_{e}}}\!\!
\tau^{*}(\mu-2)\tau^{*}(\mu+2),
\label{eq:sig2}
\end{align}
with $L_{e}$ as in Lemma~\ref{lem:levels}, and let $E$ be an integer with
$E\ge9$, to be chosen at the end. We treat the levels $e\le E$ by the
equidistribution estimate and the levels $e>E$ in a single block.

By Lemma~\ref{lem:levels} each $L_{e}$ is a union of at most two residue classes
modulo a power of $\pi$, of total measure $q_{e}$; Lemma~\ref{lem:equi}, whose
error term is uniform in the modulus and the residue, therefore gives
\begin{align}
\sum_{\substack{N\mu\le T\\ \mu\in L_{e}}}\tau^{*}(\mu-2)\tau^{*}(\mu+2)
=q_{e}\,\frac{\pi}{8G}\,T(\log T)^{2}+O\bigl(T\log T\bigr).
\label{eq:inner}
\end{align}
Summing \eqref{eq:inner} over the levels $6\le e\le E$, weighted by $W(e)$, and
using $\sum_{6\le e\le E}W(e)\ll E^{2}$ together with \eqref{eq:series},
\begin{align}
\sum_{6\le e\le E}W(e)\!\!\sum_{\substack{N\mu\le T\\ \mu\in L_{e}}}\!\!
\tau^{*}(\mu-2)\tau^{*}(\mu+2)
=\tfrac32\cdot\frac{\pi}{8G}T(\log T)^{2}
+O\Bigl(E^{2}T\log T+E2^{-E}T(\log T)^{2}\Bigr).
\label{eq:head}
\end{align}
The levels $e>E$ are bounded by Lemma~\ref{lem:tail}, and combining this with
\eqref{eq:head} we obtain, uniformly in $E\ge9$,
\begin{align}
\Sigma_{2}(T)=\tfrac32\cdot\frac{\pi}{8G}T(\log T)^{2}
+O\Bigl(E^{2}\,T\log T+E2^{-E}\,T(\log T)^{2}+E2^{-E/2}\,T(\log T)^{3}\Bigr).
\label{eq:sig2asym}
\end{align}
Choose $E=\bigl\lceil6\log\log T\bigr\rceil$, so that
$2^{-E/2}\le(\log T)^{-3\log2}$ with $3\log2=2.079\ldots>2$. The three error
terms in \eqref{eq:sig2asym} are then
$O\bigl(T\log T(\log\log T)^{2}\bigr)$, $O(T)$ and
$O\bigl(T(\log T)^{0.93}\log\log T\bigr)$ respectively, whence
\begin{align}
\Sigma_{2}(T)=\tfrac32\cdot\frac{\pi}{8G}T(\log T)^{2}
+O\bigl(T\log T(\log\log T)^{2}\bigr).
\label{eq:step2}
\end{align}

\medskip
\emph{Conclusion.} Adding \eqref{eq:step1} and \eqref{eq:step2} and multiplying
by $4$,
\begin{align}
\bigl|\B_{T}\bigr|
=4\Bigl(\tfrac32+\tfrac32\Bigr)\frac{\pi}{8G}T(\log T)^{2}
+O\bigl(T\log T(\log\log T)^{2}\bigr)
=\frac{3\pi}{2G}T(\log T)^{2}+O\bigl(T\log T(\log\log T)^{2}\bigr),
\end{align}
which is the assertion of the theorem. 
\qed

\subsection*{Proof of \thmref{Rcount}}
   First we note that the $\mathcal B_T$ is the subset, taken from $\rm SL(2,\z[\textit{i}])$. That means we need to divide by $2$ for getting elements in  $\rm PSL(2,\z[\textit{i}])$. Also, by \lemref{lem: exact}, every loxodromic reversible class with infinite cyclic
centraliser contains exactly $8$ special forms. The remaining classes are those whose
centraliser is $\z\oplus\z/3\z$; by the previous lemma these are the powers of at most
three primitive elements, so only $O(\log T)$ of them have $N\mu\le T$ and they do not
affect the asymptotic. 
   Therefore, we have
   \[ |R_T|\sim \frac{|\mathcal B_T|}{8\cdot 2}=\frac{3\pi}{32G}T(\log T)^{2}+O\bigl(T\log T(\log\log T)^{2}\bigr).\]
\qed

\section*{Declarations} 

{\bf Ethical Approval}. Not applicable.

{\bf Competing Interest}. Not applicable.

{\bf Funding}. Not applicable.

{\bf Authors' Contributions.} All authors have contributed equally. 

{\bf Availability of data and materials}. Not applicable.

\medskip {AI declaration.}
The first drafts of this paper were prepared before large language model (LLM) tools became widely available. At a later stage, the authors used Claude (Anthropic) to assist in checking parts of the proofs, tightening some arguments, and correcting an error in a previous draft. All mathematical statements and proofs were independently reviewed and verified by the authors, who take full responsibility for the content of the paper.

    \bibliographystyle{plain}
	\bibliography{ref}

\end{document}